\documentclass{amsart}
\usepackage[margin=3.8cm]{geometry}
\usepackage{amssymb}
\usepackage{mathrsfs}

\usepackage{hyperref}
\usepackage{enumitem}

\usepackage{comment}
\usepackage{graphicx}
\usepackage{color}

\usepackage[all]{xy}
\usepackage{tikz}
\usepackage{tikz-cd}

\usepackage[nobysame,abbrev]{amsrefs}

\newtheorem{theorem}{Theorem}[section]
\newtheorem{lemma}[theorem]{Lemma}
\newtheorem{prop}[theorem]{Proposition}
\newtheorem{cor}[theorem]{Corollary}

\theoremstyle{definition}
\newtheorem{definition}[theorem]{Definition}

\newtheorem{remark}[theorem]{Remark}
\newtheorem{example}[theorem]{Example}
\newtheorem{construction}[theorem]{Construction}

\numberwithin{equation}{section}

\suppressfloats

\renewcommand{\setminus}{\smallsetminus}
\renewcommand{\emptyset}{\varnothing}
\newcommand{\epsi}{\varepsilon}

\newcommand{\onto}{\twoheadrightarrow} 
\newcommand{\GL}{\mathrm{GL}} 
\newcommand{\SL}{\mathrm{SL}} 
\DeclareMathOperator{\rank}{rank}
\newcommand{\Gm}{\mathbb{G}_\mathrm{m}}
\newcommand{\al}{\alpha}
\newcommand{\chitop}{\chi_\mathrm{top}}
\newcommand{\contr}{\mathrm{contr}}
\newcommand{\tors}{\mathrm{tors}}
\newcommand{\trop}{\mathrm{trop}}

\newcommand{\bfone}{\boldsymbol{1}}
\DeclareMathOperator{\Vol}{Vol}
\newcommand{\pencilP}{\mathfrak{d}_P}
\newcommand{\minresYP}{\overline{Y}_{P}}
\newcommand{\barDP}{\overline{D}_{P}}

\DeclareMathOperator{\Hom}{Hom} 
\DeclareMathOperator{\Spec}{Spec} 
\DeclareMathOperator{\Proj}{Proj} 
\DeclareMathOperator{\Pic}{Pic}

\DeclareMathOperator{\MW}{MW} 
\DeclareMathOperator{\NS}{NS} 

\newcommand{\Newt}[1]{\mathrm{Newt}(#1)}

\newcommand{\pow}[1]{[ \!  [ #1 ] \! ]}

\newcommand\cO{\mathcal{O}}

\renewcommand\AA{\mathbb{A}}
\newcommand\CC{\mathbb{C}}
\newcommand\FF{\mathbb{F}}

\newcommand\NN{\mathbb{N}}
\newcommand\PP{\mathbb{P}}
\newcommand\QQ{\mathbb{Q}}
\newcommand\RR{\mathbb{R}}

\newcommand\ZZ{\mathbb{Z}}

\newcommand\rH{\mathrm{H}}

\newcommand\rV{\mathrm{V}}

\newcommand\rmd{\mathrm{d}}

\title{Reflexive polygons and rational elliptic surfaces}

\author[A.~Grassi]{Antonella Grassi}
\address{Dipartimento di Matematica, Universit\`a di Bologna, Piazza di Porta San Donato 5, Bologna 40126, Italy and Department of Mathematics, University of Pennsylvania, Philadelphia, USA}
\email{antonella.grassi3@unibo.it}

\author[G.~Gugiatti]{Giulia Gugiatti}
\address{Math Section, ICTP, Leonardo Da Vinci Building, Strada Costiera 11, 34151 Trieste, Italy}
\email{ggugiatt@ictp.it}

\author[W.~Lutz]{Wendelin Lutz}
\address{Department of Mathematics and Statistics, Lederle Graduate Research Tower, University of Massachusetts, Amherst, MA 01003, USA}
\email{wendelinlutz@umass.edu}

\author[A.~Petracci]{Andrea Petracci}
\address{Dipartimento di Matematica, Universit\`a di Bologna, Piazza di Porta San Donato 5, Bologna 40126, Italy}
\email{a.petracci@unibo.it}

\begin{document}

\begin{abstract}
In this note we study in detail the geometry of eight rational elliptic surfaces naturally associated to the sixteen reflexive polygons. The elliptic fibrations supported by these surfaces correspond under mirror symmetry to the eight families of smooth del Pezzo surfaces with very ample anticanonical bundle.
\end{abstract}

\maketitle

\section{Introduction}

Reflexive polygons are the lattice polygons such that the origin is their unique interior lattice point \cite{batyrev_toric_Fano_3folds, watanabe^2, rabinowitz}.
There are exactly $16$ $\GL_2(\ZZ)$-equivalence classes of reflexive polygons (see Figure~\ref{fig:reflexive_polygons}).

For each reflexive polygon $P$, we can construct two toric del Pezzo surfaces with Gorenstein singularities: $X_P$, associated to the face fan of $P$, and $Y_P$, associated to the normal fan of $P$; these are toric varieties with respect to mutually dual algebraic tori.
By performing a certain blowup of $Y_P$  (see Construction~\ref{constr:2}), we obtain a rational elliptic surface $Y \to \PP^1$, we study the singular fibres of this elliptic fibration and its Mordell--Weil group. Our results are summarised in Table~\ref{tab:summary_singular_fibre} in \S\ref{sec:conclusion}.
The number of rational elliptic surfaces arising this way from reflexive polygons is $8$, as different reflexive polygons can give the same rational elliptic surface.
The Mordell--Weil groups of these rational elliptic surfaces have low rank, precisely $0$ or $1$.

These rational elliptic surfaces are mirror to the $8$ deformation families of smooth del Pezzo surfaces with very ample anticanonical class, i.e.\ $\PP^1 \times \PP^1$ and the blowup of $\PP^2$ in at most $6$ points.

\subsection*{Notation and conventions} \label{sec:notation}

A polytope is the convex hull of finitely many points in a real vector space of finite dimension.
A polygon is by definition a polytope of dimension $2$.

All varieties we consider are varieties over $\CC$, the field of complex numbers.
Every toric variety or toric singularity is assumed to be normal.
A Fano variety is a normal projective variety whose anticanonical divisor is $\QQ$-Cartier and ample.
A del Pezzo surface is a Fano variety of dimension $2$.
By a curve we mean a $1$-dimensional integral scheme of finite type over $\CC$.

The symbol $\chitop$ stands for the topological Euler characteristic.

\subsection*{Acknowledgements}

The last three authors learnt (almost) everything contained in this note during countless conversations with Alessio Corti; they wish to thank him for encouraging them and for generously sharing his ideas about this mathematical subject.
All authors are grateful to the anonymous referee, to Victor Przyjalkowski and especially to Helge Ruddat for useful comments on a previous version of this article.

A.G.\ and A.P.\ acknowledge partial financial support from
INdAM GNSAGA ``Gruppo Nazionale per le Strutture Algebriche, Geometriche e le loro Applicazioni''. Moreover,
the work of A.G.\ is partially supported by PRIN ``Moduli and Lie Theory'' and the work of A.P.\ is partially supported by
PRIN2020 2020KKWT53 ``Curves, Ricci flat Varieties and their Interactions''.


\section{Preliminaries on rational elliptic surfaces} \label{sec:r.e.s.}

The material in this section can be found in \cite{kodaira_compact_surfaces_2, miranda_persson_extremal, miranda_torsionAquila, schuett_shioda_book}.

\subsection{Setting}

In \S\ref{sec:r.e.s.} we fix, once for all, a morphism
\[
f \colon Y \to C
\]
such that
\begin{itemize}
	\item $Y$ is a smooth projective 
 surface,
        \item $C$ is a smooth projective curve,
	\item $f$ is a relatively minimal elliptic fibration, i.e.\ $f$ is a surjective morphism with connected fibres such that the general fibre of $f$ is a curve of genus $1$ and no $(-1)$-curve of $Y$ is contained in a fibre of $f$,
	\item there exists at least a section of $f$, i.e.\ a morphism $\sigma \colon C \to Y$ such that $f \circ \sigma = \mathrm{id}_{C}$,
	\item $f$ has at least one singular fibre.
\end{itemize}
In this case, the surface $Y$ is called a \emph{(Jacobian) 
	elliptic surface}. In what follows, we will usually drop the word `Jacobian'.


\subsection{First properties}
We recall the following well-known classification of singular fibres of $f$.

\begin{theorem}[{
\cite[Theorem 6.2]{kodaira_compact_surfaces_2}}]
	Let $F$ be a fibre of $f \colon Y \to C$. Then exactly one of the following possibilities holds.
	\begin{itemize}
		\item $F$ is irreducible and reduced, and exactly one of the following possibilities holds: $F$ is a smooth elliptic curve $(I_0)$, $F$ is a nodal cubic $(I_1)$, $F$ is a cuspidal cubic ($II$).
		\item $F$ is reducible, every irreducible component of $F$ is a $(-2)$-curve (i.e.\ a smooth rational curve with self-intersection $-2$), and the configuration of the irreducible components of $F$ has one of the following types as described in Figure~\ref{fig:singular_fibres}: $I_n$ for an integer $n \geq 2$, $III$, $IV$, $I_n^*$ for an integer $n \geq 0$, $II^*$, $III^*$, $IV^*$.
	\end{itemize}
\end{theorem}

\begin{figure}
	\centering
	\def\svgwidth{0.6\textwidth}
	\smallskip
	\input{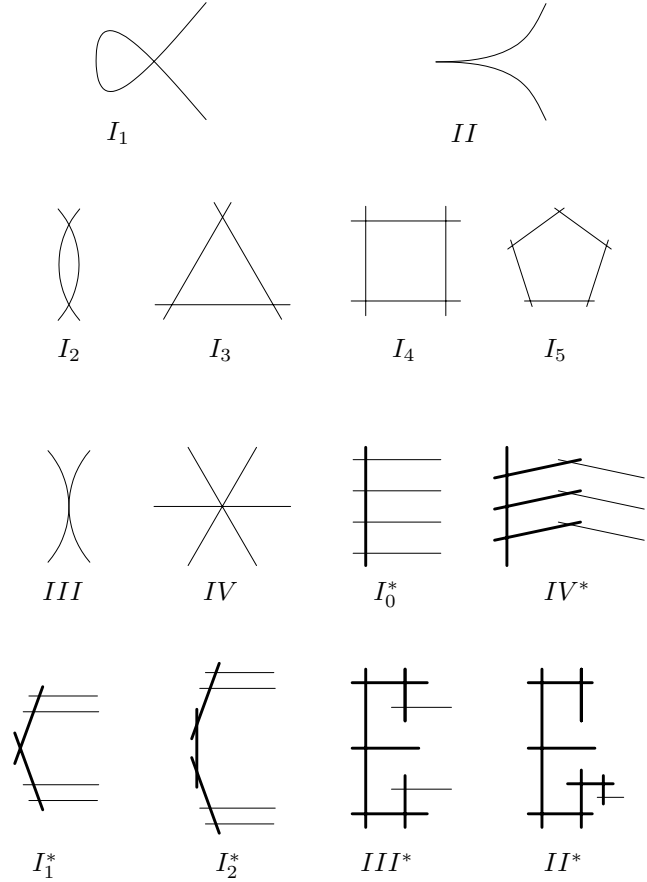}
	\caption{ 	\label{fig:singular_fibres}Some singular fibres of an elliptic fibration.
		The bold curves denote components with multiplicity $>1$.
		We have not depicted fibres of type $I_n$ for $n \geq 6$ nor of type $I^*_n$ for $n \geq 3$.}
\end{figure}

For a fibre $F$ of $f$, we define $r(F)$ to be the number of components of $F$ minus $1$, i.e.\ the number of components of $F$ which do not intersect a fixed section.
The values of $\chitop$ and of $r$ for all possible fibres are contained in Table~\ref{tab:singular_fibres}.

\begin{table}[ht!]
	\begin{equation*}
		\begin{array}{lcc}
			& \chitop & r \\
			\hline
			I_0 & 0 & 0 \\
			I_n \ (n \geq 1) & n & n-1 \\
			I_n^* \ (n \geq 0) & n+6 & n+4 \\
			II & 2 & 0 \\
			III & 3 & 1 \\
			IV & 4 & 2 \\
			IV^* & 8 & 6 \\
			III^* &9 & 7 \\
			II^* & 10 & 8 \\
		\end{array}
	\end{equation*}
	\caption{\label{tab:singular_fibres} The invariants $\chitop$ and $r$ for the possible fibres on an elliptic surface.}
\end{table}

\subsection{The Mordell--Weil lattice}

Fix a section $\sigma_0$ of $f$.
The set of sections of $f$ has the structure of a finitely generated  abelian group, in which $\sigma_0$ is the identity element. This group is denoted by $\MW(Y/C)$ or $\MW(Y)$ and is called the \emph{Mordell--Weil group} of $f$ or of $Y$. The section $\sigma_0$ is called the $0$-section.
The isomorphism class of the group $\MW(Y)$ does not depend on the choice of $\sigma_0$.

Let $R(Y)$ be the subgroup of $\NS(Y)$ generated by the components of the fibres not meeting $\sigma_0$. One has
\[
\rank R(Y) = \sum_{F \text{ singular fibre}} r(F).
\]
Let $T(Y)$ be the subgroup of $\NS(Y)$ generated by $\sigma_0$ and by the components of the fibres.
One has
\[
T(Y) = \ZZ \sigma_0 \oplus \ZZ F \oplus R(Y),
\]
where $F$ is the class of a fibre.

The following theorem states that, modulo  $T(Y)$, $\mathrm{NS}(Y)$ can be understood in terms of sections.
\begin{theorem} [{\cite[Theorem 6.5]{schuett_shioda_book}}]\label{thm:iso}
	There is an isomorphism of abelian groups
	\begin{equation*}
		\MW(Y) \simeq \NS(Y) / T(Y).
	\end{equation*}
\end{theorem}

Moreover, there exists an injective group homomorphism (see \cite[Lemma 6.17]{schuett_shioda_book})
	\begin{equation*}
		\varphi \colon \MW(Y)/{\MW(Y)_{\mathrm{tors}}} \to \NS(Y) \otimes_\ZZ \QQ,
	\end{equation*}
 which can be used to endow the Mordell-Weil group $\MW(Y)$ with the structure of a lattice. The homomorphism $\varphi$ is called the \emph{Shioda homomorphism}.
 
\begin{theorem}[{\cite[Theorem 6.20]{schuett_shioda_book}}]\label{th:HeightP}
Let $\varphi$ be the Shioda homomorphism and let $\cdot$ be the intersection product on $Y$.
	 For all sections  $\sigma_1, \sigma_2$, set:
	\begin{equation*}
		\langle \sigma_1, \sigma_2 \rangle := - \varphi (\sigma_1)\cdot \varphi (\sigma_2)
	\end{equation*}
	Then $\langle \cdot, \cdot \rangle$ is a symmetric  bilinear $\mathbb Q$-valued pairing on $\MW(Y)$ and induces the structure of a positive definite lattice on $\MW(Y)/{\MW(Y)_{\mathrm{tors}}}$.
\end{theorem}
In particular  a section $\sigma$ is torsion if and only if $\langle \sigma, \sigma \rangle =0$.

\begin{definition}[{\cite[Definition 6.21]{schuett_shioda_book}}]
	$(\MW(Y)/{{\MW(Y)}_{\mathrm{tors}}}, \langle \cdot, \cdot \rangle)$ is called the \emph{Mordell--Weil lattice} and  $\langle \cdot, \cdot \rangle$ is called the \emph{height pairing}.
\end{definition}
The following result gives an explicit formula for the height pairing. 
Given a singular fibre $F$ and two sections $\sigma_1$ and $\sigma_2$, there is a notion of \emph{local contribution}  $\mathrm{contr}_F(\sigma_1, \sigma_2)$,  which registers the way $F$ is met by $\sigma_1$ and $\sigma_2$ \cite[Definition 6.23]{schuett_shioda_book}. The  explicit description of  $\contr_{F}(\sigma_1,\sigma_2)$ will be used in this paper only for the singular fibers of type $I_n$, see Example \ref{ex:contr}.

Given a section $\sigma$ we define $\contr_F(\sigma)=\contr_F(\sigma, \sigma)$.

\begin{theorem}[{\cite[Theorem 6.24]{schuett_shioda_book}}]\label{th:HeightPExplicit} Let $\sigma_0$ be the zero section. Then for all sections $\sigma_1$ and $\sigma_2$
	
	\begin{equation*}
		\langle \sigma_1, \sigma_2 \rangle =\chi(\mathcal O_Y)  + \sigma_1 \cdot \sigma_0 + \sigma_2 \cdot \sigma_0 - \sigma_1 \cdot \sigma_2 -  \sum _{F \ \mathrm{singular \ fibre}} \contr_{F}(\sigma_1,\sigma_2),
	\end{equation*}
\end{theorem}

In this work we will use only the particular case when the fibre $F$ is of type $I_n$, which we present here:

\begin{example}[{\cite[Table 6.1]{schuett_shioda_book}}]\label{ex:contr}
	Let $F$  be of type $I_n$ and $\sigma_1$ and $\sigma_2$  intersect the $i$th  and $j$th component  respectively, for $1 \leq i \leq j \leq n-1$. Assume that we label the components cyclically from $0$ to $n-1$ and the $0$-section is labelled with $0$, then:
	\begin{equation*}
		\contr_F(\sigma_1,\sigma_2)= \frac{i(n-j)}{n}.
	\end{equation*}
\end{example}

In \cite{miranda_torsionAquila} Miranda proves the following:

\begin{prop}[{\cite[Proposition 3.1, Corollary 4.1, Corollary 4.3]{miranda_torsionAquila}}]\label{MirandaAquila}
	
	Let $Y \to \PP^1$ be an elliptic surface with semistable fibres only, that is with singular fibres of type $I_n$ only. Let $\{ I_{m_v} \}_{v \in \mathbb P^1}$ be the collection of such singular fibres. Denote the components of the $I_{m_v}$ fibre as $\{m_0 (v), m_1 (v) , \cdots, m_{m_v-1}(v) \} $. Assume that a  torsion section $\sigma \neq \sigma_0$ intersects the  fibre component $m_j (v)$.  Without loss of generality we assume $m_j (v) \leq m(v)/2.$ Then
	\begin{enumerate}
		\item $\sum_v  m_j(v) \cdot \frac{1-m_j(v)}{m_v}= 2 \chi (\mathcal O_Y)$.
		\item Let $\sigma$ be a torsion section of order $n$. Then 
		\begin{equation*}
			\sum_v  m_j(v) =
			\begin{cases}
				4 \chi (\mathcal O_Y) & \text{if  $n=2$}\\
				3 \chi (\mathcal O_Y) & \text{if  $n \geq 3$}.
			\end{cases}       
		\end{equation*} 
	\end{enumerate}
\end{prop}

In characteristic $0$, \cite[Proposition 6.33]{schuett_shioda_book}, if  $\sigma$ is a  torsion section, then  $\sigma \cdot \sigma_0=0$.

\subsection{Rational elliptic surfaces}
Suppose now that the surface $Y$ is rational.  This implies that $C=\PP^1$ \cite[$\S 7.1$]{schuett_shioda_book}. A rational elliptic surface has a very restricted topology:

\begin{prop}[{\cite[$\S 7.2$]{schuett_shioda_book}}] We have:
	\begin{enumerate}
		
		\item $\omega_Y \simeq f^* \cO_{\PP^1}(-1)$ and every fibre of $f$ is an anticanonical divisor;
		
		\item $K_Y^2 = 0$;
		
		\item $q(Y) = 0$, $p_g(Y) = 0$, $\chi(\cO_Y) = 1$;
		
		\item $\chitop (Y)= 12$, $h^{1,1}(Y) = 10$;
		
		\item		$C$ is a section of $f$ if and only if $C$ is a $(-1)$-curve in $Y$;
		
		\item $\Pic(Y) \simeq \NS(Y) \simeq \rH^2(Y, \ZZ) \simeq \ZZ^{10}$.
	\end{enumerate}
\end{prop}
It follows from Theorem \ref{thm:iso} that 

	\[
	10 = \rank \NS(Y) = \rank \MW(Y)+ \rank T (Y)
	\] and \[
	8 = \rank \MW(Y)+ \rank R (Y).
	\]
 \smallskip

Moreover, the pairing $\langle \sigma, \sigma\rangle$ simplifies. Indeed, Theorem \ref{th:HeightPExplicit}, the adjunction formula, and $\chi(\mathcal O_Y)= 1 = -\sigma^2$ give:

\begin{cor}\label{cor:HeighPP}
	For every section $\sigma$
	\begin{equation*}
		\langle \sigma, \sigma \rangle = 2 +2 \sigma \cdot \sigma_0  -  \sum _{F \ \mathrm{singular \ fibre}} \contr_{F}(\sigma,\sigma).
	\end{equation*}
\end{cor}


\subsection{Extremal rational elliptic surfaces}

Let $f \colon Y \rightarrow \PP^1$ be a rational elliptic surface.
\begin{prop}[{\cite[Introduction]{miranda_persson_extremal}}] \label{prop:extremal}
	The following are equivalent:
	\begin{enumerate}
		\item the morphism $f \colon Y \to \PP^1$ has finitely many sections;
		\item $\MW(Y)$ is a finite (abelian) group;
		\item $\rank \MW(Y)=0$;
		\item $\sum_{F \text{ singular fibre}} r(F)	=8$.
		\item the number of representations of $Y$ as a blow-up of $\PP^2$ is finite;
		\item the number of rational curves on $Y$ with negative self-intersection is finite;
		\item the number of reduced effective divisors on $Y$ with negative self-intersection is finite.
	\end{enumerate}

\end{prop}

\begin{definition}[{\cite[Introduction]{miranda_persson_extremal}}]
	The rational elliptic surface $Y$ is called \emph{extremal} if it satisfies one (and hence every) condition in Proposition~\ref{prop:extremal}.
\end{definition}

\section{Preliminaries on reflexive polygons and toric del Pezzo surfaces}

The material in this section can be found in \cite{cls}.

\subsection{Lattices}

Let $N$ be a lattice, i.e.\ a free abelian group of finite rank.
Let $M$ denote the dual lattice of $N$, i.e.\ $M = \Hom_\ZZ(N,\ZZ)$.
We consider  the $\RR$-vector spaces $N_\RR = N \otimes_\ZZ \RR$ and $M_\RR = M \otimes_\ZZ \RR = \Hom_\ZZ(N,\RR)$.
Let $\langle \cdot, \cdot \rangle \colon M \times N \to \ZZ$ denote the duality pairing; we use the same symbol to denote the duality pairing between $M_\RR$ and $N_\RR$.
We  consider the two following mutually dual algebraic tori: $T_N = \Spec \CC[M] = N \otimes_\ZZ \Gm$ and $T_M = \Spec \CC[N] = M \otimes_\ZZ \Gm$.

\subsection{Polarised toric varieties} \label{sec:polarised_toric_variety}

A polytope in $N_\RR$ is the convex hull of finitely many points of $N_\RR$.
It is a rational polytope in $N$ if its vertices are rational, i.e.\ elements of $N \otimes_\ZZ \QQ$.
A \emph{lattice polytope} in $N$ is a polytope in $N_\RR$ whose vertices are elements of $N$.
The same terminology works for polytopes in $M$.

In what follows, unless otherwise stated, every polytope will have full dimension, i.e.\ the smallest affine subspace containing the polytope is the ambient vector space itself. 

If $P$ is a full-dimensional lattice polytope in $M$, then one can consider the cone over $P$ placed at height $1$
\[
\RR_{\geq 0} (P \times \{1\}) \subseteq M_\RR \oplus \RR
\]
and the projective $T_N$-toric variety
\[
Y_P := \Proj \CC \! \left[ \left(\RR_{\geq 0} (P \times \{1\}) \right) \cap  (M \oplus \ZZ)    \right],
\]
where the $\NN$-grading is given by the projection $M \oplus \ZZ \onto \ZZ$.
The toric variety $Y_P$ is associated to the \emph{normal fan} of $P$, which is the fan in $N$ consisting of the cones orthogonal (and inward-directed) to the faces of $P$.

In addition to $Y_P$, associated to $P$ there is an ample effective Cartier divisor $D_P$ on $Y_P$. If $P$ changes by translation, then $D_P$ changes by linear equivalence, so the isomorphism class of the line bundle $L_P := \cO_{Y_P}(D_P)$ does not change. There are two links between the geometry of the polytope $P$ and the geometry of the variety $Y_P$, as follows.
\begin{enumerate}
	\item For every integer $m \geq 0$, there is a natural $1$-to-$1$ correspondence between the lattice points of the polytope $mP$ and the monomial basis of $\rH^0(Y_P, L_P)$.
	\item The top self-intersection (also called degree) of the ample line bundle $L_P$ on $Y_P$ is equal to the volume\footnote{Some authors call this the `normalised volume', but we avoid to do this.} of the polytope $P$:
	\[
	(L_P)^n = \Vol(P)
	\]
	where $n = \dim Y_P = \rank N$ and $\Vol(P)$ is equal to $n!$ times the Lebesgue measure of $P$.
\end{enumerate}

If one starts from a rational polytope in $N$ and applies the constructions described above, then one gets a projective $T_M$-toric variety together with an ample $\QQ$-Cartier $\QQ$-divisor.

Of course the roles of $M$ and $N$ could be swapped: if one starts from a lattice polytope in $N$ and applies the constructions described above then one gets a polarised projective $T_M$-toric variety.

\subsection{Toric Fano varieties}

A \emph{Fano polytope} in $N$ is a full-dimensional lattice polytope $P$ in $N$ such that the origin $0 \in N$ is in the interior of $P$ and every vertex of $P$ is a primitive lattice vector of $N$, i.e.\ there is no lattice point on the segment between the origin and every vertex.
If $P$ is a Fano polytope in $N$, then the \emph{face fan} (also called the spanning fan) of $P$ is the fan in $N$ consisting of the cones over the faces of $P$;
 we denote by $X_P$ the toric variety associated to the face fan of a Fano polytope $P$.
 We have that $X_P$ is Fano, more precisely the toric boundary (i.e.\ the reduced sum of the torus-invariant prime divisors) is anticanonical, $\QQ$-Cartier and ample.
 
 If $P$ is a Fano polytope in $N$, then the \emph{polar} of $P$ is the following rational polytope in $M$:
\begin{equation*}
	P^\circ := \{ u \in M_\RR \mid \forall v \in P, \langle u,v \rangle \geq -1  \}.
\end{equation*}
One can see that $P^\circ$ is full-dimensional and that the face fan of $P$ coincides with the normal fan of $P^\circ$; therefore $X_P = Y_{P^\circ}$.

\subsection{Reflexive polytopes}

A \emph{reflexive polytope} in $N$ is a Fano polytope $P$ in $N$ such that its polar $P^\circ$ is a lattice polytope in $M$ --- this definition dates back to \cite{batyrev_toric_Fano_3folds, watanabe^2}.
If $P$ is a reflexive polytope in $N$, then the following statements hold:
\begin{itemize}
	\item $P^\circ$ is a reflexive polytope in $M$;
	\item the toric Fano variety $X_P$ is Gorenstein, i.e.\ its canonical divisor is Cartier;
	\item for every integer $m \geq 0$, there is a natural $1$-to-$1$ correspondence between the lattice points of the polytope $mP^\circ$ and the monomial basis of $\rH^0(X_P, -mK_{X_P})$.
	\item we have an equality of polytopes $(P^\circ)^\circ = P$;
	\item we have equalities of Gorenstein toric Fano varieties:
	\[
	X_P = Y_{P^\circ} \qquad \text{and} \qquad Y_P = X_{P^\circ};
	\]
	\item the ample Cartier divisor $D_P$ on $Y_P$ associated with the polytope $P$ coincides with the toric boundary of $Y_P$, i.e.\ the reduced sum of the torus-invariant prime divisors of $Y_P$, so $D_P$ is effective, reduced and anticanonical (i.e.\ linearly equivalent to $-K_{Y_P}$);
	\item if $n = \dim N = \dim X_P = \dim Y_P$, then
	\[
	(-K_{X_P})^n = \Vol(P^\circ) \qquad \text{and} \qquad (-K_{Y_P})^n = \Vol(P);
	\]
	\item the Hilbert series of $-K_{X_P}$ is equal to the Ehrhart series of $P^\circ$, and
	the Hilbert series of $-K_{Y_P}$ is equal to the Ehrhart series of $P$.
\end{itemize}
To summarise, to every reflexive polytope $P$ one can associate two Gorenstein toric Fano varieties: $X_P$ is the one associated to the face fan of $P$, whereas $Y_P$ is the one associated to the normal fan of $P$; their big tori, namely $T_N$ and $T_M$, are dual to each other.

\subsection{Reflexive polygons} 
A reflexive polygon is a reflexive polytope of dimension $2$. There are exactly $16$ reflexive polygons, up to lattice isomorphism; these are depicted in Figure~\ref{fig:reflexive_polygons}, ordered by their volume. We will refer to them by $P_3, \dots, P_9$.
\begin{figure}[ht!]
	\centering
		\def\svgwidth{\textwidth} 
	\input{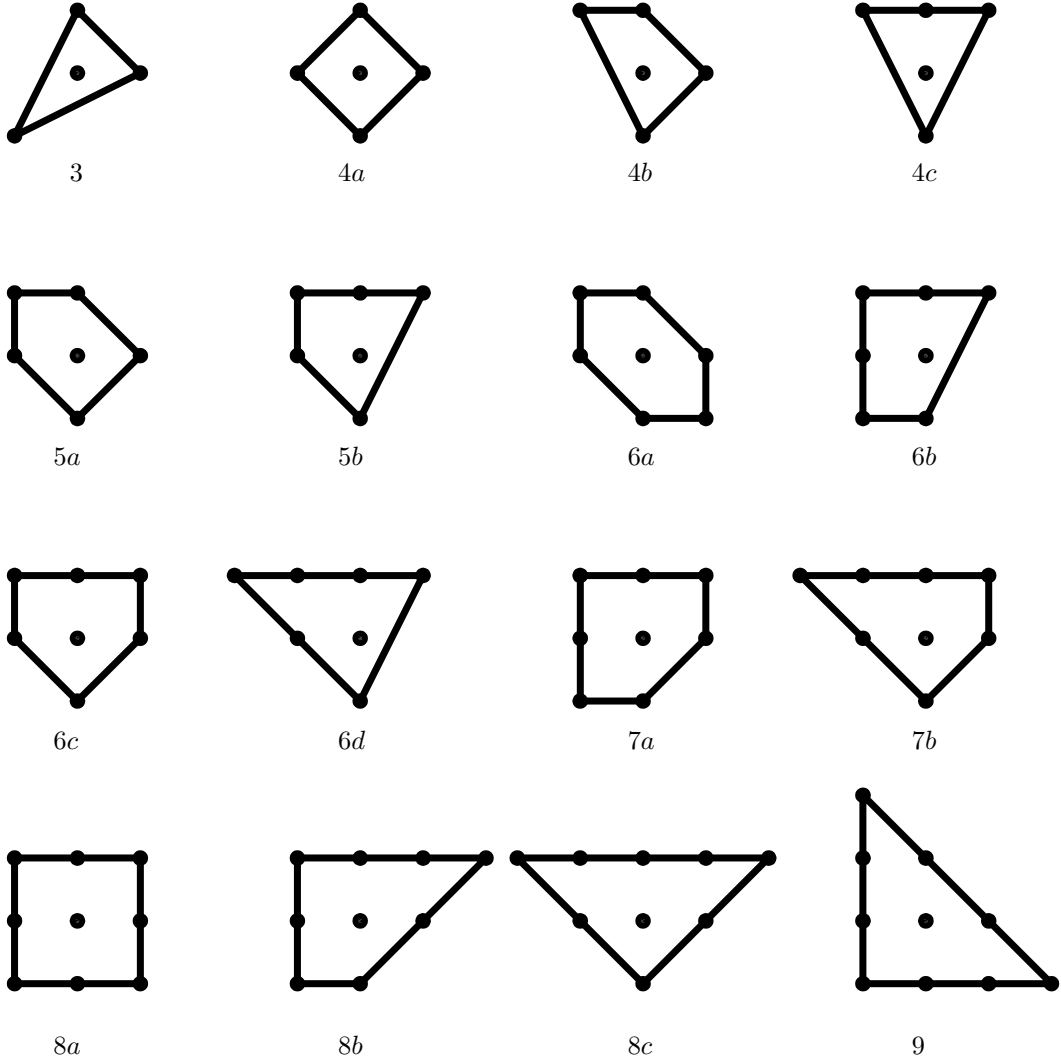}
	\caption{The 16 reflexive polygons, up to $\GL_2(\ZZ)$.}
	\label{fig:reflexive_polygons}
\end{figure}
One can check the following equalities:
\begin{itemize}
	\item $(P_3)^\circ = P_9$,
	\item $(P_{4i})^\circ = P_{8i}$ for $i = a,b,c$,
	\item $(P_{5i})^\circ = P_{7i}$ for $i = a,b$,
	\item $(P_{6a})^\circ$ is isomorphic to $P_{6a}$ via $\GL_2(\ZZ)$,
	\item $(P_{6b})^\circ$ is isomorphic to $P_{6b}$ via $\GL_2(\ZZ)$,
	\item $(P_{6c})^\circ = P_{6c}$,
	\item $(P_{6d})^\circ = P_{6d}$.
\end{itemize}
Actually, for each reflexive polygon $P$ we have
\begin{equation} \label{eq:volume_P_volume_Ppolar}
	\Vol(P) + \Vol(P^\circ) = 12;
\end{equation}
we refer the reader to \cite{reflexive_polygons_12} for more details.

\section{Rational elliptic surfaces arising from reflexive polygons} \label{sec:rational-reflexive}

Now we explain a construction which starts from a reflexive polygon $P$ and produces a rational elliptic surface $Y$ together with a non-isotrivial relatively minimal elliptic fibration $f \colon Y \to \PP^1$ with at least a section.

\begin{construction} \label{constr:1}
Fix a rank $2$ lattice $N$ and a reflexive polygon $P$ in $N$. Use the following notation as in \S\ref{sec:polarised_toric_variety}:
\begin{itemize}
	\item $Y_P$ is the (possibly singular) Gorenstein $T_M$-toric del Pezzo surface $Y_P$ associated to the normal fan of $P$;
	\item $D_P$ is the ample Cartier divisor on $Y_P$ associated with $P$, i.e.\ $D_P$ is the toric boundary of $Y_P$ and is a cycle of smooth rational curves;
	\item $L_P = \cO_{Y_P}(D_P)$ is isomorphic to $-K_{Y_P}$, which is ample.
\end{itemize}
The lattice points of $P$ correspond bijectively to the monomial basis of $\rH^0(Y_P, L_P)$, hence there is a natural $\CC$-vector space isomorphism between $\rH^0(Y_P, L_P)$ and the space of functions $P \cap N \to \CC$.
We consider two special elements of $\rH^0(Y_P, L_P)$:
\begin{itemize}
	\item $\bfone_P$ corresponds to the function $P \cap N \to \CC$ which takes the value $1$ on the origin $0 \in N$ and takes the value $0$ elsewhere;
	\item $f_P$ corresponds to the function $P \cap N \to \CC$ which takes the value $0$ on the origin and takes binomial values on the edges of $P$. This means that, for each edge $e$ of $P$ with lattice length $\ell(e)$, the values of the function  $P \cap N \to \CC$ on the ordered $\ell(e)+1$ lattice points of $e$ are the positive integers $\binom{\ell(e)}{i}$ for $0 \leq i \leq \ell(e)$.
\end{itemize}
We denote by $Z_P$ the divisor of zeroes of $f_P$; it is an anticanonical effective divisor on $Y_P$. We denote by $\pencilP$ the pencil on $Y_P$ spanned by $\bfone_P, f_P$ and we consider the corresponding rational map $Y_P \dashrightarrow \PP^1$.
\end{construction}

\begin{example}\label{ex:cubic_1}
We consider the polygon $P = P_3$ in $N = \ZZ^2$.
Then $Y_P$ is the singular cubic surface $\{ x_1 x_2 x_3 - x_0^3 = 0 \} \subset \PP^3$. There are $3$ singular points of type $A_2$: $[0:1:0:0]$, $[0:0:1:0]$, $[0:0:0:1]$.
The toric boundary $D_P$ of $Y_P$ is the union of $3$ coplanar  lines $Y_P \cap \{ x_0 = 0 \} = \{ x_1 x_2 x_3 = x_0 = 0 \}$ and the line bundle $L_P = \cO_{Y_P}(D_P)$ is isomorphic to $\cO_{\PP^3}(1) \vert_{Y_P}$.
The section $\bfone_P$ is  $x_0$, whereas the section $f_P$ is $x_1 + x_2 + x_3$.
The rational map $Y_P \dashrightarrow \PP^1$ is given by
\[
[x_0 : x_1 : x_2 : x_3] \mapsto [x_0 : x_1 + x_2 + x_3].
\]
The scheme-theoretic base locus of this pencil is reduced and supported  at the $3$ points $[0:0:1:-1]$, $[0:1:0:-1]$, $[0:1:-1:0]$ of $Y_P$.
The divisor $D_P$, the base points of $\pencilP$ and the singularities of $Y_P$ are depicted at the bottom right corner of Figure~\ref{fig:resolution_P3}.
\end{example}

Now we begin to study the base locus of the pencil $\pencilP$.

\begin{remark} \label{rmk:1}
Let $P$ be a reflexive polygon and let $Y_P, D_P, L_P, \bfone_P, f_P, Z_P, \pencilP$ be as in Construction~\ref{constr:1}.
\begin{enumerate}
	\item The divisor of zeroes of $\bfone_P$ is the toric boundary $D_P$.
	\item Fix an edge $e$ of $P$, denote by $\ell(e)$ the lattice length of $e$, and denote by $D_{P,e}$ the component of $D_P$ corresponding to $e$.
	Using toric geometry, there are two choices of an isomorphism between $D_{P,e}$ and $\PP^1$: they map the two torus-invariant points of $D_{P,e}$ to $0 = [1:0]$, $\infty = [0:1]$ and they differ by composing with the involution of $\PP^1$ given by $[x_0 : x_1] \mapsto [x_1 : x_0]$; in particular, there is a well-defined point $p_e$ on $D_{P,e}$ which corresponds to $[1:-1] \in \PP^1$. 
	The section $\bfone_P \in \rH^0(L_P)$ restricts to $0$ on $D_P$, and consequently to $0$ also on $D_{P,e}$, whereas the section $f_P \in \rH^0(L_P)$ restricts to $(x_0 + x_1)^{\ell(e)}$ on $D_{P,e}$ under both of the isomorphisms between $D_{P,e}$ and $\PP^1$ discussed above.
Therefore the effective divisors $Z_P = \{f_P = 0\}$ and $D_{P,e}$ intersect in the point $p_e$ with multiplicity $\ell(e)$. More precisely,
the scheme-theoretic intersection $Z_P \cap D_{P,e}$ is given by the ideal $(x,y^{\ell(e)})$, where $x,y$ are local coordinates of $Y_P$ at the smooth point $p_e$ such that $D_{P,e}$ is locally defined by $x=0$.

\item The effective divisors $Z_P = \{ f_P = 0 \}$ and $D_P = \sum_{e \leq P} D_{P,e} = \{ \bfone_P = 0 \}$ intersect at the points $\{ p_e \mid e \text{ edge of } P \}$ with local structure given above.

\item The base points of the pencil $\pencilP$ are disjoint from the singularities of $Y_P$.

\item For each edge $e$ of $P$, the indeterminacies at the point $p_e$ of the rational map $Y_P \dashrightarrow \PP^1$ given by the pencil $\pencilP$ are resolved by blowing up $\ell(e)$ times above $p_e$ in the proper transform of $D_P$; this introduces a chain of smooth rational curves: a $(-1)$-curve meeting transversally the proper transform of $D_{P,e}$, and $\ell(e)-1$ $(-2)$-curves.
\end{enumerate}
\end{remark}

\begin{lemma} \label{lem:connectedness_members_dP}
	Let $P$ be a reflexive polygon and let $\pencilP$ be the pencil on the surface $Y_P$, as in Construction~\ref{constr:1}.
	Then every member $V$ of $\pencilP$ is connected and such that $h^1(\cO_V) = 1$.
\end{lemma}

\begin{proof}
	Let $L_P = \cO_{Y_P}(D_P) \simeq \omega_{Y_P}^\vee$ be the line bundle on $Y_P$ associated to $P$.
	More generally, we prove that for every non-zero global section $s$ of $L_P$ the divisor $V = \{ s=0 \}$ of zeroes of $s$ is connected.
	Consider the short exact sequence
	\[
	0 \longrightarrow L_P^\vee \overset{s}{\longrightarrow} \cO_{Y_P} \longrightarrow \cO_V \longrightarrow 0
	\]
	and the corresponding long exact sequence in cohomology. We observe that $\rH^0(L_P^\vee)$ and $\rH^1(L_P^\vee)$ vanish, because they are Serre-dual to $\rH^2(\cO_{Y_P})$ and $\rH^1(\cO_{Y_P})$ which vanish by the Kodaira Vanishing Theorem.
	From the long exact sequence in cohomology we deduce that the restriction $\rH^0(\cO_{Y_P}) \to \rH^0(\cO_V)$ is an isomorphism, hence $V$ is connected.
	
	Since $\rH^2(L_P^\vee)$ is Serre-dual to $\rH^0(\cO_{Y_P})$, we have that $\rH^2(L_P^\vee)$ has dimension $1$.
	The vanishing of $\rH^1(\cO_{Y_P})$ and of $\rH^2(\cO_{Y_P})$ implies that $\rH^1(\cO_V)$ has dimension $1$.
\end{proof}

\begin{construction} \label{constr:2}
Let $P$ be a reflexive polygon and let $Y_P, D_P, \pencilP$ be as in Construction~\ref{constr:1}.
We construct $3$ projective surfaces with effective divisors on them, as follows.
\begin{itemize}
	\item $Y' \to Y_P$ is the minimal resolution of the indeterminacies of the pencil $\pencilP$ in such a way $Y'$ is smooth in a neighbourhood of the preimage of the base locus of $\pencilP$; in other words, $Y' \to Y_P$ is obtained from $Y_P$ by blowing up $\ell(e)$ times above the point $p_e$ in the proper transform of $D_P$, for all edges $e$ of $P$.
	Let $D' \subset Y'$ be the strict transform of $D_P \subset Y_P$.
	\item $\minresYP \to Y_P$ is the minimal resolution of the singularities of $Y_P$. Recall that they are DuVal singularities of type A, because $Y_P$ is Gorenstein toric of dimension $2$, so $\minresYP \to Y_P$ is crepant.
	Let $\barDP \subset \minresYP$ be the preimage of $D_P \subset Y_P$.
	\item Consider $Y = \minresYP \times_{Y_P} Y'$. Let $D \subset Y$ be the strict transform of $\barDP \subset \minresYP$ along $Y \to \minresYP$. 
\end{itemize}
We denote by
\[
f \colon Y \to \PP^1
\]
the composition of the proper birational morphism $Y \to Y_P$ with the rational map $Y_P \dashrightarrow \PP^1$ induced by the pencil $\pencilP$.
\begin{equation*}
	\xymatrix{
		Y	\ar[d] \ar[r] & Y' \ar[d] \ar@/^1.0pc/[rd] \\
		\minresYP \ar[r] & Y_P \ar@{-->}[r] & \PP^1
	}
\end{equation*}
\end{construction}

\begin{remark}
Let $P$ be a reflexive polygon in a rank-$2$ lattice $N$, let $Y_P, \pencilP$ be as in Construction~\ref{constr:1}, and let $Y', \minresYP, Y$ be as in Construction~\ref{constr:2}.

By Remark~\ref{rmk:1}(3), the base points of the pencil $\pencilP$ are disjoint from the singularities of $Y_P$.
Therefore $Y'$ and $Y_P$ have the same singularities, $Y$ is the minimal resolution of the singularities of $Y'$, the morphism $Y \to Y'$ is crepant, $Y$ is a smooth rational projective surface, and $D$ is the preimage of $D'$.
\end{remark}

\begin{example} \label{ex:cubic_2} 
We continue Example~\ref{ex:cubic_1}, so $P = P_3$.
The surface $Y_P$ is depicted at the bottom right corner of Figure~\ref{fig:resolution_P3}: the black curves denote the $3$ components of $D_P$, the red points are the $3$ $A_2$-singularities, the blue points are the $3$ base points of the pencil $\pencilP$.

\begin{figure}[ht!]
	\centering
	\def\svgwidth{0.6\textwidth}
\begingroup%
  \makeatletter%
  \providecommand\color[2][]{%
    \errmessage{(Inkscape) Color is used for the text in Inkscape, but the package 'color.sty' is not loaded}%
    \renewcommand\color[2][]{}%
  }%
  \providecommand\transparent[1]{%
    \errmessage{(Inkscape) Transparency is used (non-zero) for the text in Inkscape, but the package 'transparent.sty' is not loaded}%
    \renewcommand\transparent[1]{}%
  }%
  \providecommand\rotatebox[2]{#2}%
  \newcommand*\fsize{\dimexpr\f@size pt\relax}%
  \newcommand*\lineheight[1]{\fontsize{\fsize}{#1\fsize}\selectfont}%
  \ifx\svgwidth\undefined%
    \setlength{\unitlength}{289.83769532bp}%
    \ifx\svgscale\undefined%
      \relax%
    \else%
      \setlength{\unitlength}{\unitlength * \real{\svgscale}}%
    \fi%
  \else%
    \setlength{\unitlength}{\svgwidth}%
  \fi%
  \global\let\svgwidth\undefined%
  \global\let\svgscale\undefined%
  \makeatother%
  \begin{picture}(1,0.83094858)%
    \lineheight{1}%
    \setlength\tabcolsep{0pt}%
    \put(0,0){\includegraphics[width=\unitlength,page=1]{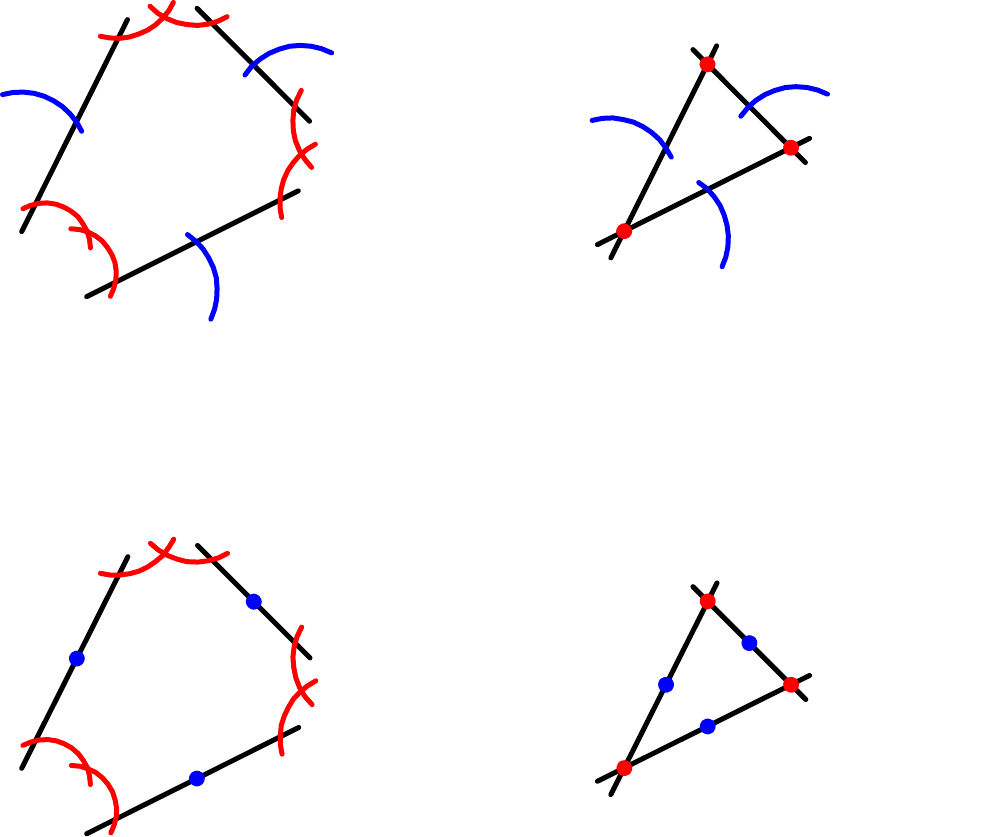}}%
    \put(0.29,0.54){\color[rgb]{0,0,0}\makebox(0,0)[lt]{\lineheight{1.25}\smash{\begin{tabular}[t]{l}$Y$\end{tabular}}}}%
    \put(0.76,0.54){\color[rgb]{0,0,0}\makebox(0,0)[lt]{\lineheight{1.25}\smash{\begin{tabular}[t]{l}$Y'$\end{tabular}}}}%
    \put(0.29,0.03){\color[rgb]{0,0,0}\makebox(0,0)[lt]{\lineheight{1.25}\smash{\begin{tabular}[t]{l}$\minresYP$\end{tabular}}}}%
    \put(0.76,0.03){\color[rgb]{0,0,0}\makebox(0,0)[lt]{\lineheight{1.25}\smash{\begin{tabular}[t]{l}$Y_P$\end{tabular}}}}%
  \end{picture}%
\endgroup%

	\caption{The surfaces constructed in Construction~\ref{constr:2}, when the reflexive polygon $P$ is $P_3$, as considered in Example~\ref{ex:cubic_1} and in Example~\ref{ex:cubic_2}.
		The red points denote the singularities of $Y_P$ and of $Y'$.
		The blue points in $Y_P$ denote the base points of the pencil $\pencilP$.
		The red curves denote the exceptional curves of $\minresYP \to Y_P$ and of $Y \to Y'$.
		The blue curves denote the exceptional curves arising from resolving the indeterminacies of $Y_P \dashrightarrow \PP^1$ (i.e.\ of the pencil $\pencilP$) and of $Y \dashrightarrow \PP^1$.
		The black segments denote the irreducible components of the toric boundary $D_P$ in $Y_P$ and their proper transforms in $\minresYP$, $Y'$ and $Y$.}
	\label{fig:resolution_P3}
\end{figure}

The surface $Y'$ is obtained by blowing up the $3$ base points, so it has $3$ $(-1)$-curves, which are depicted in blue in the top right corner of  Figure~\ref{fig:resolution_P3}. It has $3$ $A_2$-singularities.

The surface $\minresYP$ is obtained by resolving the $3$ $A_2$-singularities of $Y_P$ with $3$ chains of $2$ $(-2)$-curves; they are depicted in red in the bottom left corner of Figure~\ref{fig:resolution_P3}.
The surface $\minresYP$ is smooth.
The strict transforms in $\minresYP$ of the $3$ components of $D_P$ are $(-1)$-curves.

The surface $\minresYP$ can be constructed torically. The surface $Y_P$ is the toric variety associated to the normal fan of $P = P_3$, i.e.\ the face fan of $P^\circ = P_9$: the rays of this fan are in black in Figure~\ref{fig:cubic_fan}.
The fan of $\minresYP$ is the complete fan with the $9$ rays depicted in Figure~\ref{fig:cubic_fan}.

\begin{figure}[ht!]
	\centering
		\def\svgwidth{0.5\textwidth}
	\includegraphics{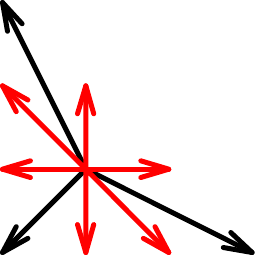}
	\caption{The rays of the fan defining the smooth toric surface $\minresYP$ when the reflexive polygon $P$ is $P_3$, as considered in  Example~\ref{ex:cubic_2}.
	The black rays correspond to the proper transforms of the irreducible components of the toric boundary of $Y_P$, whereas the red rays correspond to the exceptional curves of $\minresYP \to Y_P$.}
	\label{fig:cubic_fan}
\end{figure}

The surface $Y$, depicted in the top left corner of Figure~\ref{fig:resolution_P3} is smooth. The exceptional curves of $Y \to \minresYP$, depicted in blue, are $(-1)$-curves,
the exceptional curves of $Y \to Y'$, depicted in red, are $(-2)$-curves, the strict transforms  of the components of $D_P$, depicted in black, are $(-2)$-curves.
Therefore the divisor $D$, which is the sum of the red curves and of the black ones is a cycle of $9$ $(-2)$-curves.

\end{example}

\begin{prop} \label{prop:2}
Let $P$ be a reflexive polygon in a rank-$2$ lattice $N$, let $Y_P, D_P, \pencilP$ be as in Construction~\ref{constr:1}, and let $Y', \minresYP, Y, f$ and $D', \barDP, D$ be as in Construction~\ref{constr:2}.
Then the following statements hold.
\begin{enumerate}
	\item The divisors $D_P, D', \barDP, D$ on $Y_P, Y', \minresYP, Y$, respectively, are anticanonical.
	
	\item $\minresYP$ is the smooth toric surface associated to the complete fan in $M$ whose rays are the rays with apex at the origin and passing through all lattice boundary points of $P^\circ$.
	
	\item $\barDP$ is the toric boundary of $\minresYP$ and consists of a reduced cycle of smooth rational curves.
	
	\item $\Pic(\minresYP)$ is a free abelian group of rank equal to $10 - \Vol(P) = \Vol(P^\circ) - 2$.

	\item The exceptional locus of $Y' \to Y_P$ (resp.\ of $Y \to \minresYP$) has $\Vol(P)$ irreducible components, all of which are smooth rational curves contained in the smooth locus of $Y'$ (resp.\ $Y$) and are either $(-1)$-curves or $(-2)$-curves.

	\item $\Pic(Y)$ is a free abelian group of rank $10$.

	\item $D$ is a reduced cycle of $12 - \Vol(P)$ $(-2)$-curves and is the fibre of $f \colon Y \to \PP^1$ over $\infty = [0:1] \in \PP^1$.

	\item The morphism $f \colon Y \to \PP^1$ is a non-isotrivial relatively minimal elliptic fibration with at least a section, hence $Y$ is a rational elliptic surface.
	
	\item The Mordell--Weil group of $Y$ has cardinality $ \geq \# \{ \text{edges of } P \}$.
\end{enumerate}
\end{prop}

\begin{proof}
	Below we will freely use the equality \eqref{eq:volume_P_volume_Ppolar}.
	
	 (1) We already know that $D_P$ is an anticanonical divisor on $Y_P$. Since $\minresYP \to Y_P$ is crepant and $\barDP$ is the pull-back of $D_P$, also $\barDP$ is anticanonical.
	When we blow up a smooth point $p$, then the strict transform of an anticanonical divisor which is reduced and smooth at $p$ is anticanonical. This implies that $D'$ is anticanonical on $Y'$.
	Using the crepant morphism $Y \to Y'$ or the morphism $Y \to \minresYP$ which is a composition of blow ups of smooth points, we deduce that $D$ is anticanonical on $Y$.
	
	(2) Minimal resolutions of toric surfaces can be constructed torically.
The fan of the toric variety $Y_P$ is the normal fan of $P$, i.e.\ the face fan of $P^\circ$: its rays are the rays with apex at the origin and passing through the vertices of $P^\circ$.
The minimal resolution $\minresYP$ of $Y_P$ is associated to a refinement of this fan, namely to the fan described in (2).

(3) Obvious.

(4) The divisor $\barDP$ is the reduced sum of two types of curves: the exceptional curves $\overline{E}_j$ of $\minresYP \to Y_P$ (which are $(-2)$-curves), and the strict transforms $\overline{D}_{P,e}$ of $D_{P,e}$ with respect to $\minresYP \to Y_P$, as $e$ runs among the edges of $P$.
The number of components $\overline{E}_j$ is equal to the number of interior lattice points of the edges of $P^\circ$.
The number of components $\overline{D}_{P,e}$ is equal to the number of vertices of $P^\circ$.
In total, the number of components of $\barDP$ is equal to $\Vol(P^{\circ})$, which is also the number of rays of the fan defining $\minresYP$. 
In other words, the number of torus-invariant prime divisors on $\minresYP$  is $\Vol(P^\circ) -2$.
By \cite[Theorem~4.1.3]{cls} the rank of the divisor class group of $\minresYP$ is $\Vol(P^\circ) -2$.
Since $\minresYP$ is $\QQ$-factorial, the rank of the Picard group of $\minresYP$ is $\Vol(P^\circ) -2$.
Moreover, the Picard group of $\minresYP$ is free abelian by \cite[Proposition~4.2.5]{cls}.

Now we make a computation which will be useful below. Fix an edge $e$ of $P$. Now we compute the self-intersection of $\overline{D}_{P,e}$ in $\minresYP$. Since $D_{P,e} \subset Y_P$ is associated to the edge $e$ of lattice length $\ell(e)$ in the polytope $P$ which is given by the polarisation $D_P$, we have $D_{P,e} \cdot D_P = \ell(e)$.
Since $\pi \colon \minresYP \to Y_P$ is crepant, $D_P$ is anticanonical on $Y_P$, and $\barDP$ is anticanonical on $\minresYP$, by using projection formula we have
\begin{align*}
	\ell(e) = D_{P,e} \cdot D_P = \pi_* \overline{D}_{P,e} \cdot D_P = \overline{D}_{P,e} \cdot \pi^* D_P = \overline{D}_{P,e} \cdot \overline{D}_P = (\overline{D}_{P,e} )^2 + 2.
\end{align*}
The last equality holds because $\barDP$ is the reduced sum of a cycle of smooth curves meeting transversally and $\overline{D}_{P,e}$ is a component of $\barDP$.
Therefore
\begin{equation} \label{eq:self-intersection_DPe}
(\overline{D}_{P,e} )^2 = \ell(e) - 2.
\end{equation}
This equality will be useful below.

(5) Obvious.

(6) Combine (4) and (5).

(7) In the proper birational morphism $Y \to \minresYP$ we blow up (consecutively) $\ell(e)$ times an interior point of $\overline{D}_{P,e}$, for each edge $e$ of $P$.
Let $D_e$ denote the strict transform of $\overline{D}_{p,e}$ with respect to $Y \to \minresYP$.
This implies that $(D_e)^2 = (\overline{D}_{P,e})^2 - \ell(e)$.
Combining with \eqref{eq:self-intersection_DPe} we get that $D_e$ is a $(-2)$-curve.
Therefore $D$ is a cycle of $\Vol(P^\circ) = 12 - \Vol(P)$ $(-2)$-curves.

(8) By Lemma~\ref{lem:connectedness_members_dP} every fibre of $f \colon Y \to \PP^1$ is connected and the general one is a smooth elliptic curve.
The fibre $D = f^{-1}(\infty)$ is reducible and contains no $(-1)$-curve.
For each edge $e$ of $P$, the last exceptional curve of $Y \to \minresYP$ over the point in $\overline{D}_{P,e}$ is a $(-1)$-curve and a section of $f$.
We now prove that $f$ is relatively minimal and $Y$ is a rational elliptic surface.

By contradiction, we can contract all $(-1)$-curves of $Y$ contained in the fibre of $f$ (one by one), so we factorise $f$ as
\[
Y \overset{\epsi}\longrightarrow Y_0 \overset{f_0}{\longrightarrow} \PP^1,
\]
where $Y_0$ is smooth and projective, $\epsi$ is proper birational and not an isomorphism, and $f_0$ is relatively minimal.
It is clear that $f_0$ is a relatively minimal elliptic fibration which admits at least a section.
Hence $Y_0$ is a rational elliptic surface, in particular the Picard rank of $Y_0$ is $10$. 
But, by (6) the Picard rank of $Y$ is $10$, hence $\epsi \colon Y \to Y_0$ is an isomorphism, and this is a contradiction.

We also have that $f \colon Y \to \PP^1$ is not isotrivial (i.e.\ the smooth fibres of $f$ are not all isomorphic) because $f$ has at least a non-smooth semistable fibre (i.e.\ a fibre of type $I_n$ for some $n \geq 1$), namely $D$. Indeed, a non-smooth semistable fibre has $j$-invariant equal to $\infty$ (see for instance \cite[p.~540]{miranda_persson_extremal}), so $j$ cannot be constant in a punctured neighbourhood of the critical value corresponding to a non-smooth semistable fibre.

(9) In the proof of (8)  we have constructed a section of $f$ for each edge of $P$.
\end{proof}

\begin{remark}
	Let $P$ be a reflexive polygon and let $Y,D,f$ be as in Construction~\ref{constr:2}.
	Then $(Y,D)$ is a \emph{log Calabi--Yau pair}, i.e.\ $Y$ is smooth projective and $D$ is a simple normal crossing reduced divisor on $Y$ such that $K_Y + D$ is linearly trivial.
	Moreover, $f \colon (Y,D) \to (\PP^1, \infty)$ and $D = f^{-1}(\infty)$.
\end{remark}

\section{Analysis of the singular fibres} \label{sec:analysis}

Here we consider each of the 16 reflexive polygons $P$ and the elliptic fibration $f \colon Y \to \PP^1$ constructed in Construction~\ref{constr:2}.
For every $\lambda \in \CC$, let $F_\lambda$ denote the zero-locus of   $f_P + \lambda \bfone_P$ in $Y_P$; this is an element of $\pencilP$.
For every $\lambda \in \CC$, let $f^{-1}(\lambda) \subset Y$ denote the fibre of $f$ which contains the strict transform of $F_\lambda$, i.e.\ the fibre of $f$ over $[1: -\lambda] \in \PP^1$.
The morphism $f^{-1}(\lambda) \to F_\lambda$ is not always an isomorphism: it depends if $F_\lambda$ is non-reduced at the base points of the pencil $\pencilP$; in this case $f^{-1}(\lambda)$ can acquire certain exceptional $(-2)$-curves of $Y \to \minresYP$, with multiplicity.

We want to study the singular fibres of $f$. Recall that $D = f^{-1}(\infty)$, the fibre at infinity, is a singular fibre of type $I_{12-\Vol(P)}$; in particular $\chitop(D) = 12 - \Vol(P)$, therefore the sum of the topological Euler characteristics of the remaining singular fibres of $f$ is equal to $\Vol(P)$.

\subsection{3} \label{sec:3}
Set $P = P_3$.
The divisor $F_\lambda$ on $Y_P$ is given by the equations $x_0^3 = x_1 x_2 x_3$, $x_1  + x_2 + x_3 + \lambda x_0 = 0$ in $\PP^3$, so $F_\lambda$ is isomorphic to the plane curve
\[
\{  x_0^3 + x_1 x_2 (\lambda x_0 + x_1 + x_2) = 0  \} \subset \PP^2.
\]
The partial derivatives of this cubic polynomial are
\begin{gather*}
	3 x_0^2 + \lambda x_1 x_2, \\
	x_2 (\lambda x_0 + 2 x_1 + x_2), \\
	x_1 (\lambda x_0 + x_1 + 2 x_2).
	\end{gather*}
The singularities of $F_\lambda$ are given by 
\[
\left\{ 
\begin{array}{l}
x_1 = 0 \\
\lambda x_0 + 2 x_1 + x_2 = 0 \\
3 x_0^2 +\lambda x_1 x_2 = 0
\end{array}
\right.
\! , \
\left\{ 
\begin{array}{l}
	\lambda x_0 + x_1 + 2 x_2 = 0 \\
	 x_2 = 0 \\
	3 x_0^2 +\lambda x_1 x_2 = 0
\end{array}
\right.
\! , \
\left\{ 
\begin{array}{l}
	\lambda x_0 + x_1 + 2 x_2 = 0 \\
	\lambda x_0 + 2 x_1 + x_2 = 0 \\
	3 x_0^2 +\lambda x_1 x_2 = 0
\end{array}
\right.
\]
which are equivalent to
\[
\left\{
\begin{array}{l}
	x_1 = x_2 \\
	\lambda x_0 + 3 x_1 = 0 \\
	3 x_0^2 + \lambda x_1^2 = 0
\end{array}
\right.
\]
i.e.\ to
\[
\left\{
\begin{array}{l}
	x_1 = x_2 \\
	3x_1 = - \lambda x_0 \\
	(27 + \lambda^3) x_0^2 = 0
\end{array}
\right.
\! .
\]
Let $\zeta \in \CC$ be a primitive $3$rd root of unity.
Then $F_\lambda$ is smooth for $\lambda \neq -3, -3 \zeta, -3 \zeta^2$.

Let us analyse the singularities of $F_{-3} \subset \PP^2$. The unique singular point of $F_{-3} \subset \PP^2$ is $[1:1:1] \in \PP^2$.
We dehomogeneise the polynomial defining $F_{-3}$ using affine coordinates $x = (x_1-x_0)/x_0$, $y = (x_2-x_0)/x_0$ (i.e.\ $x_0=1$, $x_1 = x+1$, $x_2 = y+1$): we get the polynomial 
\[
1 + (x+1)(y+1)(-3+x+1+y+1) = x^2 + xy  +y^2 + x^2 y + x y^2,
\]
whose quadratic part $x^2 + xy  +y^2$ is non-degenerate. Therefore $F_{-3}$ is a nodal irreducible cubic curve, i.e.\ a fibre of type $I_1$.
In a similar way one can prove that $F_{-3 \zeta}$ and $F_{-3 \zeta^2}$ are fibres of type $I_1$.

To sum up, the fibration $Y \to \PP^1$ has one $I_9$ fibre and $3$ $I_1$ fibres.

 We now apply the above analysis to the study of the Mordell--Weil group and lattice:
\begin{remark}
 Let $\sigma_0$, $\sigma_3$, $\sigma_6$ be the three sections, intersecting the $0$th, third, and sixth  component of $I_9$ respectively. Let  $\sigma_0$ be the $0$-section.  Following  Example~\ref{ex:contr}  we find:

\begin{align*} 
	\contr_{I_9} (\sigma_3) &=\frac{3(9-3)}{9}=2\\
				\contr_{I_9} (\sigma_6) &= \frac{6(9-6)}{9}=2 
				\end{align*}
The Shioda  homomorphism and the height pairing in  Corollary~\ref{cor:HeighPP}
give:
\[
\langle \sigma_3, \sigma_3 \rangle = \langle \sigma_6, \sigma_6 \rangle = 0,
\]
hence the sections are torsion, by Theorem~\ref{th:HeightP}.  Let $n$ be the order of  $\MW(Y)_{\tors}$. Then $n^2$ must divide the determinant of the trivial sublattice $T$,  where $ \det T= \det (A_8)=9$ (see for example \cite[Proposition~6.3.1]{schuett_shioda_book}). Then $\MW(Y)_{\tors} \simeq \mathbb Z /3 \mathbb Z$. 
	
\end{remark}


\subsection{4a}

Let $P=P_{4a}$.  	
Label the lattice points of $P$ as follows: $x_0=(0,0)$, $x_1=(1,0)$, $x_2=(0,1)$, $x_3=(-1,0)$, $x_4=(0,-1)$.
Then the toric surface $Y_P$ is given by:
\begin{equation*}
	\left\{  x_0^2-x_1x_3=x_0^2-x_2x_4=0 \right\} \subset \PP^4
\end{equation*} 
It has $4$ $A_1$ singularities.
The toric boundary $D_P \subset Y_P$ is   $Y_P \cap \{ x_0=0 \} \subset \PP^4$.
The section $\bfone_P$ is $x_0$ and the section $f_P$ is $x_1+x_2+x_3+x_4$,   thus the base locus of $\pencilP$ is given by the four points
$[0:0:0:1:-1]$, $[0:0:1:-1:0]$,  $[0:1:0:0:-1]$, 
$[0:1:-1:0:0]$.

The minimal resolution $\minresYP$ contains $4$  $(-2)$-curves.
The strict transforms in $\minresYP$ of the $4$ components of $D_P$ are $(-1)$-curves. 

The exceptional curves of $Y \to \overline{Y}_P$ are $4$ $(-1)$ curves. We draw them in blue in Figure~\ref{fig:4a}. The divisor $D=f^{-1}(\infty) \subset Y$ is a $I_8$ fibre. 
The curve $F_\lambda=\lambda \textbf{1}_P+ f_P$ is cut out of $Y_P$ by the equation $\lambda x_0+x_1 +x_2 +x_3+x_4$, thus it is isomorphic to the curve:
\begin{equation}
	\left\{ x_0^2-x_1x_3=x_0^2+x_2(\lambda x_0+x_1+x_2+x_3)=0 \right\} \subset \PP^3	
\end{equation}
The affine patch $F_\lambda \cap \{ x_3 \neq 0 \}$ is isomorphic to the affine curve $C_\lambda$:
\begin{equation} \label{eq:P4a-C-lambda}
	\left( x^2+y(\lambda x+x^2+y+1)=0\right)\subset  \AA^2 	
\end{equation}
The curve $C_\lambda$ is singular at a point $p$ if and only if $p$ satisfies the $2$ equations:
\begin{equation} \label{eq:system}
	\begin{split}
		&2x(1+y)+\lambda y=0,\\
		& x^2+\lambda x+2y+1=0.
	\end{split}
\end{equation}
By the second equation of \eqref{eq:system}, it must be
\begin{equation}
	\label{eq:y-4a}
	y=-\frac{1}{2}(x^2+\lambda x+1)
\end{equation}
The resultant of the two polynomials in $x$ that one gets by replacing $y$ with \eqref{eq:y-4a} in \eqref{eq:P4a-C-lambda} and in the first equation of \eqref{eq:system}
is $-\tfrac{1}{64}\lambda^2(\lambda-4)(\lambda+4)$.
We have that $C_0=((1+y)(x^2+y)=0) \subset \AA^2$. The two components of $C_0$ intersect transversely at 
$(\pm 1, -1)$. 
Thus $f^{-1}(0) \simeq F_0$ is a $I_2$ fibre.
The curve $C_{\pm4}$ has a node at $(\mp 1,1)$.

The sum of the topological Euler characteristic of the singular fibres  different from the $I_8$ fibre must be 4. It follows that $f$ has singular fibres of type $I_8, I_2, I_1, I_1$.

\begin{figure}[ht!]
	\tikzset{every picture/.style={line width=0.5pt}} 
	\begin{tikzpicture}[x=0.75pt,y=0.75pt,yscale=-1,xscale=1, scale=0.7]
		 \path (200,100); 
		
		\draw    (220.5,160) -- (270.5,211.75) ;
		\draw    (360,160.75) -- (310.5,212) ;
		\draw    (309.5,70.25) -- (359.5,119.75) ;
		\draw    (270,70.25) -- (221,119) ;
		\draw [color={rgb, 255:red, 208; green, 2; blue, 2 }  ,draw opacity=1 ]   (253.5,212.5) .. controls (293.5,182.5) and (305.5,197.25) .. (330,211.75) ;
		\draw [color={rgb, 255:red, 208; green, 2; blue, 2 }  ,draw opacity=1 ]   (361.5,101.25) .. controls (350,115.75) and (332.5,160.25) .. (362,172.75) ;
		\draw [color={rgb, 255:red, 208; green, 2; blue, 2 }  ,draw opacity=1 ]   (256,65) .. controls (278,88.75) and (290,95.25) .. (330,65.25) ;
		\draw [color={rgb, 255:red, 208; green, 2; blue, 2 }  ,draw opacity=1 ]   (213,102.75) .. controls (242,121.25) and (235.5,146.25) .. (222,173.75) ;
		\draw [color={rgb, 255:red, 2; green, 46; blue, 208 }  ,draw opacity=1 ]   (225,83.25) .. controls (253,87.75) and (255.5,105.25) .. (258,110.25) ;
		\draw [color={rgb, 255:red, 2; green, 46; blue, 208 }  ,draw opacity=1 ]   (316,169.25) .. controls (338,178.75) and (345,187.75) .. (347.5,201.75) ;
		\draw [color={rgb, 255:red, 2; green, 46; blue, 208 }  ,draw opacity=1 ]   (229,188.5) .. controls (233,178.75) and (250,166.75) .. (261.5,166.75) ;
		\draw [color={rgb, 255:red, 2; green, 46; blue, 208 }  ,draw opacity=1 ]   (308,104) .. controls (331,85.75) and (343,90.25) .. (356,95.25) ;
	\end{tikzpicture}
	\caption{\label{fig:4a} The surface $Y$ when $P=P_{4a}$.}	
\end{figure}
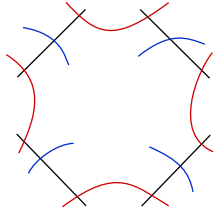

\smallskip

 We now apply the above analysis to the study of the Mordell--Weil group and lattice:
\begin{remark}
	Let $\sigma_0$, $\sigma_2$, $\sigma_4, \sigma_6$ be the 4 sections described above, with  $\sigma_j$ intersecting the $j$th component of $I_8$. Let  $\sigma_0$ be the $0$-section.  
	Tate--Shioda  formula implies that $\rank (\MW(Y))=0$, hence the sections are torsion, by Theorem~\ref{th:HeightP}.  Let $n$ be the order of  $\MW(Y)_{\tors}$. Since  $n^2$ must divide the determinant of the trivial sublattice $T$,  we have  that $n^2$ must divide $16= 2 \times 8 $ and $\MW(Y)=\MW(Y)_{\tors} \simeq \mathbb Z /4 \mathbb Z$ (Proposition~\ref{MirandaAquila}). 
	
\end{remark}


\subsection{4b}
Set $P = P_{4b}$.
Label the lattice points of $P$ as follows: $x_0 = (0,0)$, $x_1 = (1,0)$, $x_2 = (0,1)$, $x_3 = (-1,1)$, $x_4 = (-1,0)$.
This implies that $Y_P$ is the complete intersection of the quadrics $x_0^2 - x_2 x_4 = 0$ and $x_1 x_3 - x_0 x_2 = 0$ in $\PP^4$.
The divisor $F_\lambda \subset Y_P$ is given by the equation $\lambda x_0 + x_1 + x_2 + x_3 + x_4 = 0$.
Therefore $F_\lambda$ is isomorphic to the intersection of the quadrics $x_0^2 + x_2 (\lambda x_0 + x_1 + x_2 + x_3) = 0$ and $x_1 x_3 - x_0 x_2 = 0$ in $\PP^3$.
We need to understand when $F_\lambda$ is singular.

We work in the affine chart $\{ x_2 \neq 0 \}$. We use the coordinates $x = x_1/x_2$, $y = x_3/x_2$. Hence $x_0 = xy x_2$ and $x_4 = x^2 y^2 x_2$.
This implies that $F_\lambda \cap \{ x_2 \neq 0 \}$ is isomorphic to the affine curve
\[
C_\lambda = \{ 1+x+y + \lambda xy + x^2 y^2 =0 \} \subset \CC^2
\]

We analyse when $C_\lambda$ is singular. The singular points of $C_\lambda$ must satisfy the equations
\[
\left\{
\begin{array}{l}
1+x+y + \lambda xy + x^2 y^2  = 0 \\
1 + \lambda y + 2 x y^2 = 0 \\
1 + \lambda x + 2 x^2 y = 0
\end{array}
 \right.
\]
which are equivalent to 
\[
\left\{
\begin{array}{l}
	1+x+y + \lambda xy + x^2 y^2  = 0 \\
	1 + \lambda y + 2 x y^2 = 0 \\
	(\lambda + 2xy)(x-y) = 0
\end{array}
\right.
.
\]
We have two cases: (i) $xy = -\lambda / 2$ and (ii) $x=y$.
\begin{itemize}
	\item[(i)]  If $xy = -\lambda /2$, then the equations become
	\[
	\left\{
	\begin{array}{l}
		1+x+y - \frac{1}{4} \lambda^2 = 0 \\
		1 + 2 \lambda y = 0 \\
		1 + 2 \lambda x = 0
	\end{array}
	\right.
	;
	\]
	subtracting the second equation from the third one we get $\lambda (x-y) = 0$. It is quite clear $\lambda$ cannot be zero, otherwise the equations are impossible. Therefore we must have $x=y$, which is a case treated in (ii).
	\item[(ii)] If $x=y$, then the equations become
\[
\left\{
\begin{array}{l}
	1+ 2x + \lambda x^2 + x^4  = 0 \\
	1 + \lambda x + 2 x^3 = 0 \\
 x= y
\end{array}
\right.
.
\]
The resultant of the first two polynomials is $g = \lambda^4 - \lambda^3 - 8 \lambda^2 +36 \lambda - 11$, whose discriminant is $-22665187 \neq 0$. Therefore $g$ has $4$ distinct roots in $\CC$.
\end{itemize}
Hence there are exactly $4$ values of $\lambda \in \CC$ such that $C_\lambda$ is a singular curve.
Since $C_\lambda$ is an open subset of $F_\lambda$, we have found at least $4$ values of $\lambda \in \CC$ such that $F_\lambda$ (and consequently $f^{-1}(\lambda))$ is a singular curve.

We already know that there is a $I_8$ fibre. By comparing the topological Euler characteristic we have that there are no more than $4$ further singular fibres. Hence the $4$ singular fibres we have found must have topological Euler characteristic equal to $1$, thus they must be of type $I_1$.

To sum up, the singular fibres of $f$ are one of type $I_8$ and $4$ of type $I_1$.

 We now apply the above analysis to the study of the Mordell--Weil group and lattice:
\begin{remark}
	Let $\sigma_0$, $\sigma_2$, $\sigma_5, \sigma_7$ be the four sections described above,  with $\sigma_j$ intersecting the $j$th  component of $I_8$. Let  $\sigma_0$ be the $0$-section.  Following  Example~\ref{ex:contr}  we find:
	
	\begin{align*} 
		\contr_{I_8} (\sigma_2) &=\frac{2(8-2)}{8}=\frac{3}{2}\\
		\contr_{I_8} (\sigma_5) &=\frac{5(8-5)}{8}=\frac{15}{8}\\
		\contr_{I_8} (\sigma_7) &=\frac{7(8-7)}{8}=\frac{7}{8}
	\end{align*}
	 and:
	 	\begin{align*} 
	 	\contr_{I_8} (\sigma_2, \sigma _5) &=\frac{2(8-5)}{8}=\frac{3}{4}\\
	 	\contr_{I_8} (\sigma_5, \sigma_7) &=\frac{5(8-7)}{8}=\frac{5}{8}\\
	 	\contr_{I_8} (\sigma_7, \sigma_2) &=\frac{2(8-7)}{8}=\frac{2}{8}.
	 \end{align*}
 The sections $\sigma_2$, $\sigma_5$, $\sigma_7$ are not torsion, by Theorem~\ref{th:HeightP}. The  Shioda  homomorphism and the height pairing in  Corollary \ref{cor:HeighPP}
	give the height matrix:
\begin{equation}
\frac{1}{8}
	\begin{bmatrix}
	4 & 2 &6\\
	2& 1 & 3\\
	6 & 3 & 9
	\end{bmatrix}.
\end{equation}
The matrix has rank $1$, as predicted by the Tate--Shioda formula.
Moreover, since $8$ appears as a denominator in $\langle \sigma_5, \sigma_5 \rangle = \frac{1}{8}$, it follows from the classification of Mordell--Weil lattices of rational elliptic surfaces (see \cite[Theorem~8.8]{schuett_shioda_book}) that  the Mordell--Weil lattice with the height pairing must be $\langle \frac{1}{8} \rangle$, hence there is no torsion in $\MW(Y)$.
\end{remark}


\subsection{4c}
Let $P=P_{4c}$. The normal fan of $P$ is the face fan of $P_{8c}$. 
The toric surface $Y_P$ is the quotient  $\PP(1,1,2)/_{{\ZZ/2\ZZ} (0,1,1)}$. It is isomorphic to:
\begin{equation}
	\left\{ x_2^2-x_1x_3	=x_0^2-x_2x_4=0 \right\} \subset \PP^4
\end{equation}
The surface $Y_P$ has two $A_1$ singularities, at $[0 : 0 : 0 : 1 : 0]$ and at $[0 : 1 : 0 : 0 : 0]$, and one $A_3$ singularity, at $[0 : 0 : 0 : 0 : 1]$.
The toric boundary $D_P \subset Y_P$ is   $Y_P \cap \{ x_0=0 \} \subset \PP^4$.
The sections $\textbf{1}_P$, $f_P$ are  $x_0$, $x_1+2x_2+x_3+x_4$,   thus the base locus of $\pencilP$ is given by $4$ points: $[0:0:0:1:-1]$, $[0:1:0:0:-1]$, $[0:1:-1:1:0]$ and an infinitely near basepoint of the first order at $[0:1:-1:1:0]$.

The minimal resolution $\overline{Y}_P$ has $5$  $(-2)$-curves.
The strict transforms in $\overline{Y}_P$ of the three components of $D_P$ are two $(-1)$-curves, and a $0$-curve.

The exceptional curves of $Y \to \overline{Y}_P$ are $3$ $(-1)$ curves, and a $(-2)$-curve. In Figure~\ref{fig:4c}, we draw the $(-1)$-curves in blue, the $(-2)$-curve in cyan.
The divisor $D=f^{-1}(\infty) \subset Y$ is a $I_8$ fibre.
The curve $F_\lambda= \lambda \textbf{1}_P+ f_P$ is cut out of $Y_P$ by the equation $\lambda x_0+x_1 +2x_2 +x_3+x_4$, thus it is isomorphic to the curve
\begin{equation}
	\left\{ x_2^2-x_1x_3=x_0^2+x_2(\lambda x_0+x_1+2x_2+x_3)=0\right\} \subset \PP^3.
\end{equation}
The curve $F_0$ has a node at $[0:1:-1:1:0]$, thus $f^{-1}(0)$ is a $I_2$ fibre (the union of the purple curve and the cyan curve in Figure \ref{fig:4c}). The curve $F_{\pm 4}$ has a node at $[\mp 2 : 1 : 1 : 1]$, thus $f^{-1}(\pm4) \simeq F_{\pm4}$  is a $I_1$ fibre.

The sum of the topological Euler characteristic of the singular fibres  different from the $I_8$ fibre must be $4$.
It follows that $f$ has singular fibres of type $I_8$, $I_2$, $I_1$, $I_1$. 


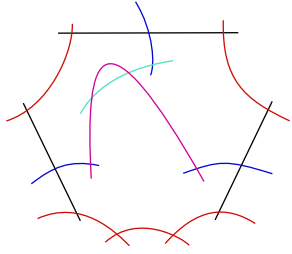
\begin{figure}[ht!]
\tikzset{every picture/.style={line width=0.5pt}} 

\begin{tikzpicture}[x=0.75pt,y=0.75pt,yscale=-1,xscale=1, scale=0.9]
	
	\draw    (250.5,140.25) -- (282.5,206.25) ;
	\draw    (390,139.25) -- (358,205.75) ;
	\draw    (270,101) -- (371,100.75) ;
	\draw [color={rgb, 255:red, 208; green, 13; blue, 2 }  ,draw opacity=1 ]   (330,219) .. controls (348,198.75) and (364,197.25) .. (380.5,207.75) ;
	\draw [color={rgb, 255:red, 208; green, 13; blue, 2 }  ,draw opacity=1 ]   (258.5,205) .. controls (283.5,193.75) and (302.5,212.75) .. (310,220.25) ;
	\draw [color={rgb, 255:red, 208; green, 13; blue, 2 }  ,draw opacity=1 ]   (296.5,218.25) .. controls (315,201.5) and (340,216) .. (343.5,220) ;
	\draw [color={rgb, 255:red, 208; green, 13; blue, 2 }  ,draw opacity=1 ]   (278,95.75) .. controls (278.5,111.75) and (265,142.75) .. (241,150.25) ;
	\draw [color={rgb, 255:red, 208; green, 13; blue, 2 }  ,draw opacity=1 ]   (362.6,93.8) .. controls (361.8,127.4) and (374.2,138.2) .. (399.8,150.2) ;
	\draw [color={rgb, 255:red, 2; green, 10; blue, 208 }  ,draw opacity=1 ]   (340.1,179.6) .. controls (360.6,172.35) and (364.2,171.8) .. (390.1,179.85) ;
	\draw [color={rgb, 255:red, 2; green, 10; blue, 208 }  ,draw opacity=1 ]   (255,185.5) .. controls (267,176.2) and (276.2,171.8) .. (293,175.25) ;
	\draw [color={rgb, 255:red, 2; green, 10; blue, 208 }  ,draw opacity=1 ]   (313.4,83.4) .. controls (323.8,100.2) and (324.3,119.6) .. (321.8,124.6) ;
	\draw [color={rgb, 255:red, 80; green, 227; blue, 194 }  ,draw opacity=1 ]   (282.5,146.25) .. controls (287,137.25) and (301,121.25) .. (334.5,116.5) ;
	\draw [color={rgb, 255:red, 208; green, 2; blue, 156 }  ,draw opacity=1 ]   (288.6,182.6) .. controls (284.6,109.8) and (295,83.8) .. (351.8,184.2) ;

\end{tikzpicture}
\caption{\label{fig:4c}The surface $Y$ when $P=P_{4c}$.}
\end{figure}

\subsection{5a}

Let $P=P_{5a}$.
The normal fan of $P$ is the face fan of $P_{7a}$. Then the toric surface $Y_P$ is the weighted blow up of $\PP^2$ at $[0:1:0]$ and $[0:0:1]$ with weights $(1,2)$. It has two $A_1$ singularities. The base locus of the pencil $\pencilP$ is formed by $5$ points. 

The minimal resolution $\overline{Y}_P$ contains two $(-2)$-curves. The strict transforms of the component of $D_P$ in $\overline{Y}_P$ are all $(-1)$-curves. 

The exceptional curves of $Y \to \overline{Y}_P$ are five $(-1)$-curves (in blue in Figure~\ref{fig:5a}). The divisor $D=f^{-1}(\infty)$ is a $I_7$ fibre. 
The affine curve $C_\lambda=F_\lambda \cap U_\sigma$, where $\sigma$ is the cone of the normal fan of $P$ spanned by $e_1$ and $-e_2$, is given by
\begin{equation} \label{eq:poly-5a}
	\left\{ \lambda xy+x^2y+x+1+y+y^2x =0\right\} \subset \AA^2.
\end{equation}
The curve $C_\lambda$ is singular at a point $p$ if and only if $p$ satisfies the two equations:
\begin{equation}
	\label{eq:system-5a}
	\begin{split}
		&\lambda y+ 2xy +1+  y^2 =0,\\
		& \lambda x+ x^2 + 1+ 2xy =0.
	\end{split}
\end{equation}
By the first equation above we obtain
\begin{equation} \label{eq:x-5a}
	x=-\frac{\lambda y +1+  y^2}{2 y}
\end{equation}
The resultant of the two polynomials in $y$ that one gets by replacing $x$ with \eqref{eq:x-5a} in \eqref{eq:poly-5a} and the second equation of \eqref{eq:system-5a} and clearing the denominators is
\[
(\lambda - 1)^2 (\lambda^3-\lambda^2-18\lambda+43).
\]
The curve $C_1$ is $\{ (x+y+1)(xy+1)=0 \} \subset \AA^2$. The two components of $C_1$ intersect transversely.

The sum of the topological Euler characteristic of the singular fibres different from the $I_7$ fibre is $5$.
It follows that	
$f$ has fibres of type $I_7, I_2, I_1, I_1,I_1$.

\begin{figure}[ht!]
	\tikzset{every picture/.style={line width=0.5pt}} 
	
	\begin{tikzpicture}[x=0.75pt,y=0.75pt,yscale=-1,xscale=1, scale=0.7]
		
		\draw    (206,118.5) -- (282.5,193.75) ;
		\draw    (289,52) -- (360,122.75) ;
		\draw    (220,48.25) -- (220,145.25) ;
		\draw    (206,59.5) -- (307.5,60.25) ;
		\draw    (390.7,145.1) -- (310.3,222.7) ;
		\draw [color={rgb, 255:red, 208; green, 2; blue, 2 }  ,draw opacity=1 ]   (357.5,107.75) .. controls (343.5,136.25) and (369.4,158.6) .. (398.6,149.8) ;
		\draw [color={rgb, 255:red, 208; green, 2; blue, 2 }  ,draw opacity=1 ]   (264.6,185.3) .. controls (313,189.8) and (318.6,200.2) .. (317,228.2) ;
		\draw [color={rgb, 255:red, 2; green, 3; blue, 208 }  ,draw opacity=1 ]   (327,177) .. controls (334,184.5) and (344,198.5) .. (378,194.5) ;
		\draw [color={rgb, 255:red, 2; green, 3; blue, 208 }  ,draw opacity=1 ]   (303,92.5) .. controls (326,99.5) and (344,94.5) .. (358,79.5) ;
		\draw [color={rgb, 255:red, 2; green, 3; blue, 208 }  ,draw opacity=1 ]   (259,36.5) .. controls (263,46.5) and (263,63.5) .. (256,77.5) ;
		\draw [color={rgb, 255:red, 2; green, 3; blue, 208 }  ,draw opacity=1 ]   (224,167.5) .. controls (250,171.5) and (264,154.5) .. (267,149.5) ;
		\draw [color={rgb, 255:red, 2; green, 3; blue, 208 }  ,draw opacity=1 ]   (192,88.5) .. controls (213,100.5) and (218,99.5) .. (241,99.5) ;

	\end{tikzpicture}
	
	\caption{ \label{fig:5a} The surface $Y$ when $P=P_{5a}$.}
\end{figure}
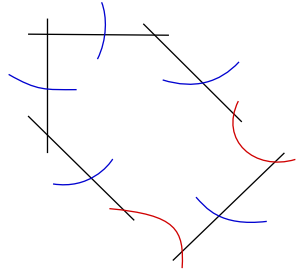

\subsection{5b}

Let $P=P_{5b}$. The normal fan of $P$ is the face fan of $P_{7b}$. Then $Y_P$ is the weighted blow up of $\PP(1,1,2)$ at $[0:1:0]$ with weights $(1,3)$. It has one $A_1$ singularity and one $A_2$ singularity. The base locus of the pencil $\pencilP$ is given by $5$ points, one of which is infinitely near of the first order. 

The minimal resolution $\overline{Y}_P$ of $Y_P$ contains three $(-2)$-curves. The strict transform in $Y_P$ of the four components of $D_P$ are three $(-1)$-curves and a $0$-curve. 

The exceptional curves of $Y \to \overline{Y}_P$ are four $(-1)$-curves and a $(-2)$-curve. The divisor $D=f^{-1}(\infty)$ is a $I_7$ fibre. 

The affine curve $C_\lambda=F_\lambda \cap U_\sigma$, where $\sigma$ is the cone of the normal fan of $P$ spanned by $e_1$ and $-e_2$, is given by
\begin{equation} \label{eq:poly-5b}
	\left\{ \lambda xy+y+xy^2+1+2x+x^2=0\right\} \subset \AA^2.
\end{equation}
The curve $C_1$ has a node at the point $(-1,0)$.  By \eqref{eq:poly-5b} this point is where there is the infinitely near base point of $\pencilP$ of the first order.
Then $f^{-1}(1)$ is a $I_2$ fibre (the union of the purple and the cyan curve in Figure~\ref{fig:5b}). 
The curve $C_\lambda$, where $\lambda$ is a root of $\lambda^3 - \lambda^2 - 18\lambda + 43$, is nodal.

The sum of the topological Euler characteristic of the singular fibres  different from the $I_7$ fibre must be 5.
It follows that $f$ has singular fibres of type $I_7$, $I_2$, $I_1$, $I_1$, $I_1$.  

\begin{figure}[ht!]

\tikzset{every picture/.style={line width=0.5pt}} 

\begin{tikzpicture}[x=0.75pt,y=0.75pt,yscale=-1,xscale=1]

\draw    (400.6,110.8) -- (361.4,190.4) ;
\draw    (274.67,73.83) -- (274.33,106.8) -- (274.33,140.55) ;
\draw    (270,132.5) -- (314.2,175.2) ;
\draw    (263.03,80.27) -- (373,80.83) ;
\draw [color={rgb, 255:red, 208; green, 27; blue, 2 }  ,draw opacity=1 ]   (291.8,164) .. controls (314.6,158) and (334.6,173.2) .. (338.2,180.4) ;
\draw [color={rgb, 255:red, 208; green, 27; blue, 2 }  ,draw opacity=1 ]   (365,75.83) .. controls (369,86.8) and (396.6,115.2) .. (407,117.6) ;
\draw [color={rgb, 255:red, 2; green, 10; blue, 208 }  ,draw opacity=1 ]   (351,146.83) .. controls (358.33,145.83) and (378.6,147.6) .. (385,157.2) ;
\draw [color={rgb, 255:red, 2; green, 10; blue, 208 }  ,draw opacity=1 ]   (276.8,161.6) .. controls (281,148.8) and (288.2,139.6) .. (301,140.8) ;
\draw [color={rgb, 255:red, 2; green, 10; blue, 208 }  ,draw opacity=1 ]   (269,109.5) .. controls (280.67,110.5) and (287.67,111.17) .. (299.4,117.47) ;
\draw [color={rgb, 255:red, 2; green, 10; blue, 208 }  ,draw opacity=1 ]   (335.67,70.83) .. controls (343,86.17) and (338.4,93.97) .. (336,103.17) ;
\draw [color={rgb, 255:red, 80; green, 227; blue, 194 }  ,draw opacity=1 ]   (344.33,96.83) .. controls (327.67,96.83) and (298.33,99.83) .. (287,105.17) ;
\draw [color={rgb, 255:red, 208; green, 2; blue, 156 }  ,draw opacity=1 ]   (358.2,153.2) .. controls (315.4,65.2) and (284.67,75.83) .. (296.67,148.83) ;
\draw [color={rgb, 255:red, 208; green, 24; blue, 2 }  ,draw opacity=1 ]   (324.8,175.4) .. controls (353.8,170.8) and (369.6,183.4) .. (378.6,188) ;

\end{tikzpicture}

	\caption{\label{fig:5b} The surface $Y$ when $P=P_{5b}$.}
\end{figure}
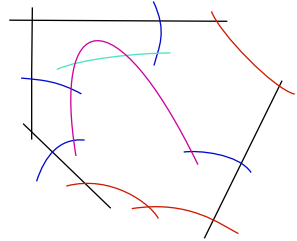

\subsection{6a}
Set $P = P_{6a}$. Then $Y_P = \minresYP$ is the smooth del Pezzo surface of degree $6$, i.e.\ the blow-up of $\PP^2$ at $3$ distinct points.
The toric boundary of $Y_P$ is made up of $6$ $(-1)$-curves, which becomes an $I_6$ fibre in $Y$.

Let us use the affine chart $U$ of $Y_P$ isomorphic to $\AA^2$ associated to the cone with rays $e_2$ and $-e_1$.
Then $F_\lambda \cap U$ is isomorphic to
\[
C_\lambda = \left\{  1+x+y+ \lambda xy + x^2 y + x y^2 + x^2 y^2 = 0  \right\} \subset \AA^2.
\]

$C_3$ is reducible with equation $(1+x+xy)(1+y+xy) = 0$. The two irreducible components of $C_3$ intersect transversally away from the toric boundary $D_P$. Therefore $f^{-1}(3)$ is an $I_2$ fibre.

$C_2$ is reducible with equation $(1+x)(1+y)(1+xy) = 0$. Each pair of the three irreducible components of $C_2$ intersects transversally in one distinct point. Therefore $f^{-1}(2)$ is an $I_3$ fibre.

With a topological Euler characteristic count, we see that there must be also an $I_1$ fibre.

\begin{figure}[ht!]
	\centering
	\includegraphics[width=0.65\textwidth]{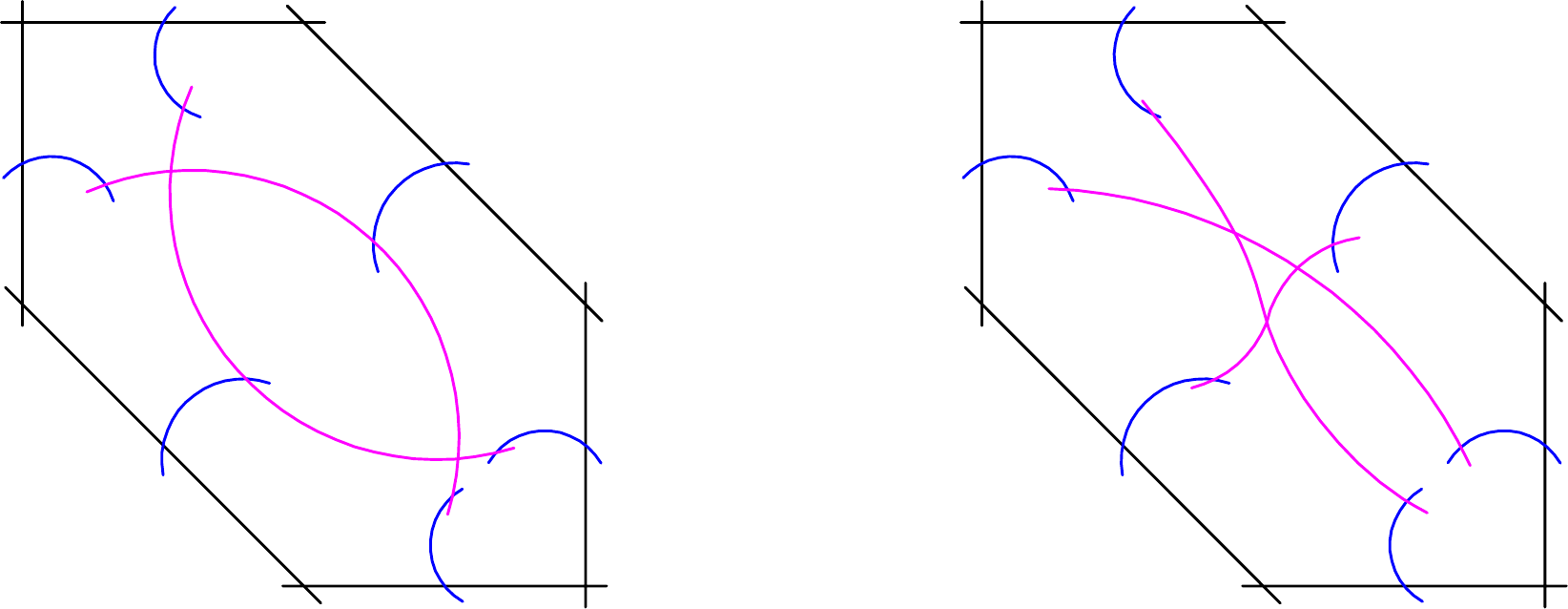}
	\caption{The surface $Y$ when $P = P_{6a}$.}
	\label{fig:P6a}
\end{figure}

\subsection{6b} Let $P=P_{6b}$. The toric surface $Y_P$ is the weighted blow-up of $\PP(1,1,2)$ at $[0:1:0]$ with weights $(1,2)$. It has two $A_1$ singularities. The base locus of the pencil $\pencilP$ is made up of $6$ points, two of which are infinitely near of the first order.

The minimal resolution $\overline{Y}_P$  of $Y_P$ contains two $(-2)$-curves. The strict transforms of the components of $D_P$ in $\overline{Y}_P$ are two $(-1)$-curves, and two $0$-curves. 

The exceptional curves of $Y \to  \overline{Y}_P$ are four $(-1)$-curves and two $(-2)$-curves. In Figure~\ref{fig:6b} we draw the $(-1)$-curves in blue, the $(-2)$-curves in cyan. The divisor $D=f^{-1}(\infty)$ is a $I_6$ fibre.  

The affine curve $C_\lambda=F_\lambda \cap U_\sigma$, where $\sigma$ is the cone of the normal fan of $P$ spanned by $e_1$ and  $-e_2$, is given by
\begin{equation}
	\label{eq;poly-6b}	
	\left\{ \lambda xy+2y+1+2x+x^2+y^2+y^2x=0 \right\} \subset \AA^2.
\end{equation}

For $\lambda  = 2$ we get
\begin{equation*}
	C_2 = \left\{ (x+1)(x+y^2+2y+1)=0 \right\} \subset \AA^2;
\end{equation*}
the component $\{x+y^2+2y+1=0\}$ of $C_2$ meets the other component $\{x+1=0\}$ transversely at $(-1,0)$ and $(-1,-2)$.
By \eqref{eq;poly-6b} the point $(-1,0)$ is where we have an infinitely near base point of $\pencilP$ of the first order.
Then $f^{-1}(2)$ is a $I_3$ fibre (the triangle formed by the two purple curves and one of the cyan curves in Figure~\ref{fig:6b}).

The curve $C_3$ has a node at $(0,-1)$.  By \eqref{eq;poly-6b} the point $(0,-1)$ is where we have the other infinitely near base point of $\pencilP$ of the first order.  Then $f^{-1}(3)$ is a $I_2$ fibre (the union of the yellow curve and the other cyan curve in Figure~\ref{fig:6b}). 

The sum of the topological Euler characteristic of the singular fibres different from the $I_6$ fibre must be $6$. It follows that $f$ has fibres of type $I_6, I_3, I_2, I_1$.  

\begin{figure}[ht!]
\tikzset{every picture/.style={line width=0.5pt}} 

\begin{tikzpicture}[x=0.75pt,y=0.75pt,yscale=-1,xscale=1, scale=1.2]
	
	\draw    (241.14,114.34) -- (240.67,215.17) ;
	\draw    (232,119.83) -- (332.37,119.79) ;
	\draw    (229.33,210.5) -- (295.95,208.93) ;
	\draw    (319.51,222.21) -- (364.5,128.65) ;
	\draw [color={rgb, 255:red, 208; green, 27; blue, 2 }  ,draw opacity=1 ]   (322.73,111.48) .. controls (328.62,131.41) and (342.54,141.93) .. (372,135.84) ;
	\draw [color={rgb, 255:red, 208; green, 27; blue, 2 }  ,draw opacity=1 ]   (281,212.83) .. controls (312.33,193.17) and (331.06,221.84) .. (332.67,223.5) ;
	\draw [color={rgb, 255:red, 2; green, 3; blue, 208 }  ,draw opacity=1 ]   (316.58,174.75) .. controls (345.51,166.72) and (351.93,175.03) .. (353.54,176.69) ;
	\draw [color={rgb, 255:red, 2; green, 3; blue, 208 }  ,draw opacity=1 ]   (269.67,227.5) .. controls (267.33,223.5) and (264.33,205.17) .. (266.67,190.5) ;
	\draw [color={rgb, 255:red, 2; green, 3; blue, 208 }  ,draw opacity=1 ]   (282.6,102.6) .. controls (279.4,117) and (280.38,122.56) .. (288.2,139) ;
	\draw [color={rgb, 255:red, 80; green, 227; blue, 194 }  ,draw opacity=1 ]   (288.2,131.4) .. controls (275.8,149.4) and (302.46,170.39) .. (304.6,172.6) ;
	\draw [color={rgb, 255:red, 208; green, 2; blue, 142 }  ,draw opacity=1 ]   (263,197.17) -- (297.67,161.5) ;
	\draw [color={rgb, 255:red, 2; green, 3; blue, 208 }  ,draw opacity=1 ]   (226.2,166.2) .. controls (235,171.8) and (245.06,172.38) .. (260.06,167.96) ;
	\draw [color={rgb, 255:red, 80; green, 227; blue, 194 }  ,draw opacity=1 ]   (252.6,166.6) .. controls (258.38,171.09) and (268.2,175) .. (284.2,166.2) ;
	\draw [color={rgb, 255:red, 208; green, 160; blue, 2 }  ,draw opacity=1 ]   (261.67,178.17) .. controls (272.67,141.5) and (278.67,165.5) .. (282.33,170.83) ;
	\draw [color={rgb, 255:red, 208; green, 2; blue, 142 }  ,draw opacity=1 ]   (247.33,174.17) .. controls (257,187.83) and (287.67,184.5) .. (288.67,182.83) ;
	\draw [color={rgb, 255:red, 208; green, 164; blue, 2 }  ,draw opacity=1 ]   (258,193.5) .. controls (256.33,213.5) and (336.67,203.83) .. (331,162.83) ;
	\draw [color={rgb, 255:red, 208; green, 164; blue, 2 }  ,draw opacity=1 ]   (265.67,129.17) .. controls (295.33,113.5) and (328.67,145.83) .. (331,162.83) ;
	\draw [color={rgb, 255:red, 208; green, 164; blue, 2 }  ,draw opacity=1 ] [dash pattern={on 0.84pt off 2.51pt}]  (261.67,178.17) -- (258,193.5) ;
	\draw [color={rgb, 255:red, 208; green, 2; blue, 142 }  ,draw opacity=1 ]   (247.33,174.17) .. controls (232,149.83) and (313.67,136.83) .. (320.67,179.83) ;

\end{tikzpicture}
	\caption{\label{fig:6b} The surface $Y$ when $P=P_{6b}$.}	
\end{figure}
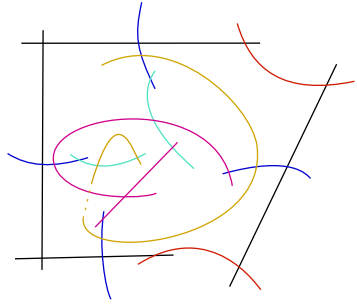

\subsection{6c}

Set $P=P_{6c}$. The toric surface $Y_P$ has one $A_1$-singularity. Resolving, the toric boundary of $\minresYP$ has six components, giving rise to an $I_6$ fibre over $\infty$. The base locus of the pencil consists of six basepoints, one of which is infinitely near of the first order.
$\minresYP$ has a smooth affine patch in which $F_\lambda$ is given by the equation 
\[
x^2y+x^2+2x+1+y+y^2x+\lambda xy=0.
\]
By observing that $P$ admits two different Minkowski decompositions into two pieces, we see that (the affine patches) of $F_2$ and of $F_3$ are reducible, more precisely given by the equations
\[
(1 + x + y) (1 + x + x y) = 0 \quad \text{and} \quad (1 + y) (1 + 2 x + x^2 + x y) = 0,
\]
respectively.

By checking the other charts, we verify that the singular locus of $F_3$ is disjoint from the toric boundary, so that $F_3$ gives rise to a $I_2$ fibre. 
On the other hand, the two branches of $F_2$ meet the toric divisor $\{y=0\}$ transversely at the point $(-1,0)$. The general member of $F_\lambda$ is tangent to order $2$ to $\{y=0\}$ at $(-1,0)$, so that after blowing up the base locus, $F_2$ picks up an exceptional divisor (in cyan in Figure~\ref{fig:6c}), and becomes a $I_3$ fibre (the purple curves in Figure~\ref{fig:6c} are the proper transforms of the two components of $F_2$).

To find the other singular fibres, we can solve $\frac{\partial F}{\partial x}=0$ for $x$, and substitute into $F=\frac{\partial F}{\partial y}=0$. The resultant of $F$ and $\frac{\partial F}{\partial y}$ is a polynomial in $\lambda$ with roots at $\lambda=2,3$ and $-6$. 
Since the Euler numbers of the singular fibres must add up to $12$, we must have an $I_1$ fibre at $\lambda=-6$.

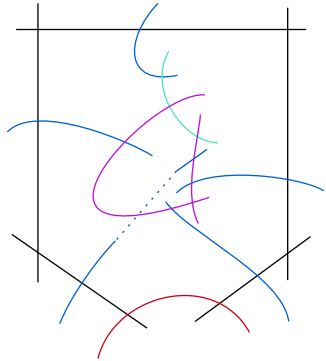
\begin{figure}[ht!]

\tikzset{every picture/.style={line width=0.5pt}} 

\begin{tikzpicture}[x=0.75pt,y=0.75pt,yscale=-1,xscale=1, scale=1.1]
	
	\draw    (318.5,36) -- (318.5,164) ;
	\draw    (443.5,143) -- (390.5,182) ;
	\draw    (433,38) -- (433,163) ;
	\draw    (441.5,48) -- (308.5,48) ;
	\draw    (368.5,185) -- (306.5,142) ;
	\draw [color={rgb, 255:red, 208; green, 2; blue, 27 }  ,draw opacity=1 ]   (346,199) .. controls (353.5,169) and (399.5,158) .. (415.5,187) ;
	\draw [color={rgb, 255:red, 0; green, 97; blue, 208 }  ,draw opacity=1 ]   (373.5,36) .. controls (357.5,53) and (358.5,74) .. (382.5,69) ;
	\draw [color={rgb, 255:red, 80; green, 227; blue, 194 }  ,draw opacity=1 ]   (378.5,58) .. controls (368,76) and (386,100) .. (401,100) ;
	\draw [color={rgb, 255:red, 0; green, 97; blue, 208 }  ,draw opacity=1 ]   (304.5,95) .. controls (318,81) and (362,100) .. (371,106) ;
	\draw [color={rgb, 255:red, 0; green, 97; blue, 208 }  ,draw opacity=1 ]   (382,123) .. controls (395.5,109) and (440.5,116) .. (449.5,122) ;
	\draw [color={rgb, 255:red, 0; green, 97; blue, 208 }  ,draw opacity=1 ]   (377,127) .. controls (388,142) and (430.5,161) .. (433.5,182) ;
	\draw [color={rgb, 255:red, 0; green, 97; blue, 208 }  ,draw opacity=1 ]   (328.5,183) .. controls (334,169) and (347,152) .. (354,145) ;
	\draw [color={rgb, 255:red, 189; green, 16; blue, 224 }  ,draw opacity=1 ]   (397,125) .. controls (293,160) and (369,76) .. (395,78) ;
	\draw [color={rgb, 255:red, 189; green, 16; blue, 224 }  ,draw opacity=1 ]   (393,87) .. controls (392,99) and (385,123) .. (392,137) ;
	\draw [color={rgb, 255:red, 0; green, 88; blue, 193 }  ,draw opacity=1 ]   (396,103) -- (382,113) ;
	\draw [color={rgb, 255:red, 0; green, 80; blue, 181 }  ,draw opacity=1 ] [dash pattern={on 0.84pt off 2.51pt}]  (382,113) -- (354,145) ;

\end{tikzpicture}

	\caption{\label{fig:6c} The surface $Y$ when $P = P_{6c}$.}
	
\end{figure}

\subsection{6d}

Set $P=P_{6d}$. The toric surface $Y_P$ has one $A_1$-singularity and one $A_2$-singularity.
Resolving, the toric boundary of $\minresYP$ has six components, giving rise to an $I_6$ fibre over $\infty$. The base locus of the pencil consists of $6$ basepoints, two of which are infinitely near of the first order, and one which is infinitely near of the second order. $\minresYP$ has a smooth affine patch in which $F_\lambda$ is given by the equation 
\[
x^3+3x^2+3x+1+2y+y^2+\lambda xy=0
\]
To find the singular fibres, we can solve $\frac{\partial F}{\partial y}=0$ for $y$, and substitute into $F=\frac{\partial F}{\partial x}=0$. The resultant of $F$ and $\frac{\partial F}{\partial x}$ is a polynomial in $\lambda$ with roots at $\lambda=2$,$3$ and $-6$. 
We check that $F_2, F_3, F_{-6}$ are irreducible curves.

The curve $F_2$ has a node at the point $(-1,0)$, with principal tangents $\{ y = 0\}$ (which is a component of the toric boundary $D_P \subset Y_P$) and $\{2x+y+2=0\}$ (which is transverse to the toric boundary $D_P \subset Y_P$).
Since at the point $(-1,0)$ there are an infinitely near base point of the pencil of the first order and an infinitely near base point of the pencil of the second order,  we have that $f^{-1}(2)$ is the union of the strict transform of $F_2$ (in purple in the middle of Figure~\ref{fig:6d}) and of two $(-2)$-curves which lie over the point $(-1,0)$ (in cyan in the middle of Figure~\ref{fig:6d}), hence $f^{-1}(2)$ an $I_3$ fibre.

The curve $F_3$ has a node at the point $(0,-1)$, with principal tangents given by the equation $\al x - y - 1 = 0$, where $\al \in \CC$ is such that $\al^2 + 3 \al +3 =0$.
Therefore the fibre $f^{-1}(3)$ is the union of the proper transform of $F_3$ (in yellow in Figure~\ref{fig:6d}) and the $(-2)$-curve over the point $(0,-1)$ (in cyan on the left of Figure~\ref{fig:6d}), hence it is an $I_2$-fibre.

The curve $F_{-6}$ is smooth along the base locus of the pencil and must be singular somewhere in $Y_P \setminus D_P$.
Since the Euler numbers of the singular fibres must add up to $12$, we must have a $I_1$ fibre at $\lambda=-6$.

\begin{figure}[ht!]

	\tikzset{every picture/.style={line width=0.5pt}} 

\begin{tikzpicture}[x=0.75pt,y=0.75pt,yscale=-1,xscale=1, scale=1]
	
	\draw    (238,56) -- (327.5,217) ;
	\draw    (464.5,122) -- (421.5,216) ;
	\draw    (461.5,68) -- (232.5,68) ;
	\draw [color={rgb, 255:red, 208; green, 2; blue, 27 }  ,draw opacity=1 ]   (366,219) .. controls (373.5,189) and (419.5,178) .. (435.5,207) ;
	\draw [color={rgb, 255:red, 0; green, 97; blue, 208 }  ,draw opacity=1 ]   (370.5,52) .. controls (354.5,69) and (350,85) .. (378,99) ;
	\draw [color={rgb, 255:red, 80; green, 227; blue, 194 }  ,draw opacity=1 ]   (371.5,81) .. controls (349.5,108) and (351.5,124) .. (374.5,144) ;
	\draw [color={rgb, 255:red, 0; green, 97; blue, 208 }  ,draw opacity=1 ]   (268.5,141) .. controls (282,127) and (327,133) .. (336,139) ;
	\draw [color={rgb, 255:red, 0; green, 97; blue, 208 }  ,draw opacity=1 ]   (425,109) .. controls (412,137) and (422,174) .. (450.5,191) ;
	\draw [color={rgb, 255:red, 80; green, 227; blue, 194 }  ,draw opacity=1 ]   (332,177) .. controls (312,168) and (321,117) .. (334,119) ;
	\draw [color={rgb, 255:red, 208; green, 2; blue, 27 }  ,draw opacity=1 ]   (312,225) .. controls (319.5,195) and (365.5,184) .. (381.5,213) ;
	\draw [color={rgb, 255:red, 208; green, 2; blue, 27 }  ,draw opacity=1 ]   (465.5,139) .. controls (443.5,117) and (441.5,83) .. (460.5,59) ;
	\draw [color={rgb, 255:red, 80; green, 227; blue, 194 }  ,draw opacity=1 ]   (354.5,130) .. controls (386.5,142) and (405,121) .. (417,102) ;
	\draw [color={rgb, 255:red, 207; green, 193; blue, 1 }  ,draw opacity=1 ]   (341.48,87.03) .. controls (364,66) and (436,78) .. (445,109) .. controls (454,140) and (408,159) .. (314,172) ;
	\draw [color={rgb, 255:red, 189; green, 16; blue, 224 }  ,draw opacity=1 ]   (291,138) .. controls (326,91) and (402,94) .. (428,134) ;
	\draw [color={rgb, 255:red, 207; green, 193; blue, 1 }  ,draw opacity=1 ] [dash pattern={on 0.84pt off 2.51pt}]  (324,118) .. controls (323,105) and (325,100) .. (341.48,87.03) ;
	\draw [color={rgb, 255:red, 207; green, 193; blue, 1 }  ,draw opacity=1 ]   (324,118) -- (337,130) ;

\end{tikzpicture}

\caption{\label{fig:6d} The surface $Y$ when $P = P_{6d}$.}
\end{figure}
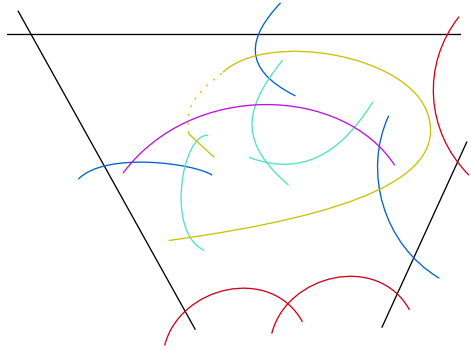

\subsection{7a}

Set $P=P_{7a}$. The toric surface $Y_P$ is smooth, so $\minresYP = Y_P$.
The toric boundary gives rise to an $I_5$ fibre over $\infty$. The base locus of the pencil consists of $7$ basepoints, two of which are infinitely near of the first order. $\minresYP$ has a smooth affine patch in which $F_\lambda$ is given by the equation 
\[
y^2+xy^2+x^2y+x^2+2x+1+2y+\lambda xy=0
\]
To find the singular fibres, we can solve $\frac{\partial F}{\partial y}=0$ for $y$, and substitute into $F=\frac{\partial F}{\partial x}=0$. The resultant of $F$ and $\frac{\partial F}{\partial x}$ is a polynomial in $\lambda$ with roots at $\lambda=3$ and $\tfrac{5}{2}(-1 \pm \sqrt{5})$. 
Observing that $P$ admits a Minkowski decomposition into three pieces, we can factor 
\[
F_3=(1 + x) (1 + y) (1 + x + y)
\]
as a triangle of lines singular at $(0,-1), (-1,0)$ and $(-1,-1)$. 
At $(0,-1)$, the curve $\{F_3=0\}$ meets the toric divisor $\{x=0\}$ transversely, with the tangents to the two branches being distinct, and similarly at $(-1,0)$. 
Since the general member of $F_\lambda$ is tangent to order $2$ at $(0,-1)$ and $(-1,0)$, we have that, after blowing up the base locus, $F_3$ picks up two exceptional curves, so that the fibre of $Y$ corresponding to $F_3$ is of type $I_5$.
An Euler number count now shows that the other two singular fibres are of type $I_1$.

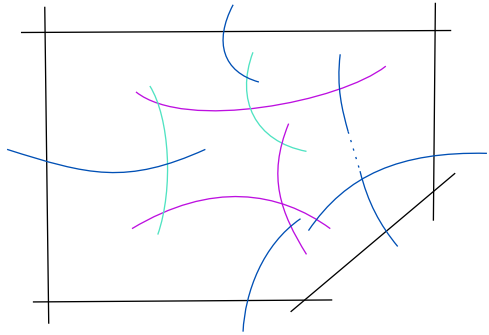
\begin{figure}[ht!]
	\tikzset{every picture/.style={line width=0.5pt}} 
	
	\begin{tikzpicture}[x=0.75pt,y=0.75pt,yscale=-1,xscale=1]
		
		\draw    (243,37) -- (245,197) ;
		\draw    (231,50) -- (448,49) ;
		\draw    (440,35) -- (439,145) ;
		\draw    (237,186) -- (388,185) ;
		\draw    (450,121) -- (367,191) ;
		\draw [color={rgb, 255:red, 189; green, 16; blue, 224 }  ,draw opacity=1 ]   (289,80) .. controls (314,100) and (392,85) .. (415,67) ;
		\draw [color={rgb, 255:red, 80; green, 227; blue, 194 }  ,draw opacity=1 ]   (348,60) .. controls (339,84) and (349,105) .. (375,110) ;
		\draw [color={rgb, 255:red, 189; green, 16; blue, 224 }  ,draw opacity=1 ]   (287,149) .. controls (327,125) and (361,130) .. (387,149) ;
		\draw [color={rgb, 255:red, 0; green, 80; blue, 181 }  ,draw opacity=1 ]   (338,36) .. controls (328,55) and (333,70) .. (351,75) ;
		\draw [color={rgb, 255:red, 80; green, 227; blue, 194 }  ,draw opacity=1 ]   (296,77) .. controls (306,93) and (308,129) .. (300,152) ;
		\draw [color={rgb, 255:red, 189; green, 16; blue, 224 }  ,draw opacity=1 ]   (366,96) .. controls (355,123) and (362,141) .. (375,162) ;
		\draw [color={rgb, 255:red, 0; green, 80; blue, 181 }  ,draw opacity=1 ]   (343,201) .. controls (344,181) and (353,157) .. (372,144) ;
		\draw [color={rgb, 255:red, 0; green, 80; blue, 181 }  ,draw opacity=1 ]   (376,150) .. controls (395,122) and (422,110) .. (466,111) ;
		\draw [color={rgb, 255:red, 0; green, 80; blue, 181 }  ,draw opacity=1 ]   (402,120) .. controls (406,136) and (412,147) .. (421,158) ;
		\draw [color={rgb, 255:red, 0; green, 80; blue, 181 }  ,draw opacity=1 ]   (392,61) .. controls (390,78) and (393,89) .. (396,100) ;
		\draw [color={rgb, 255:red, 0; green, 80; blue, 181 }  ,draw opacity=1 ] [dash pattern={on 0.84pt off 2.51pt}]  (396,100) -- (402,120) ;
		\draw [color={rgb, 255:red, 0; green, 80; blue, 181 }  ,draw opacity=1 ]   (224,109) .. controls (265,122) and (284,127) .. (324,109) ;

	\end{tikzpicture}

\caption{\label{fig:7a} The surface $Y$ when $P = P_{7a}$.}

\end{figure}

\subsection{7b}

Set $P=P_{7b}$. The toric surface $Y_P$ has one $A_1$-singularity. Resolving, the toric boundary of $\minresYP$ has five components, giving rise to an $I_5$ fibre over $\infty$. The base locus of the pencil consists of seven basepoints, two of which are infinitely near of the first order, and one which is infinitely near of the second order. The toric surface has a smooth affine patch in which $F_\lambda$ is given by the equation 
\[
x^3+3x^2+3x+1+2y+y^2+x^2y+\lambda xy=0
\]
To find the singular fibres, we can solve $\frac{\partial F}{\partial y}=0$ for $y$, and substitute into $F=\frac{\partial F}{\partial x}=0$. The resultant of $F$ and $\frac{\partial F}{\partial x}$ is a polynomial in $\lambda$ with roots at $\lambda=3$ and $\tfrac{5}{2}(-1 \pm \sqrt{5})$. 
Observing that $P$ admits a Minkowski decomposition into two pieces, we can factor 
\[
F_3=(1 + x + y) (1 + 2 x + x^2 + y)
\]
as a union of two curves meeting at $(0,-1), (-1,0)$.
At $(0,-1)$, the two components fo $F_3$ meet the toric divisor $\{x=0\}$ transversely, with the tangents to the two curves being distinct.
Since the point $(0,-1)$ is where we have an infinitely near base point of first order, $F_3$ picks up one exceptional divisor here when resolving the base locus. 
On the other hand, at $(-1,0)$, one curve meets the divisor $\{y=0\}$ transversely, and the other curve is tangent to $\{y=0\}$. The point $(-1,0)$ is where we have an infinitely near base point of first order and an infinitely near base point of second order, so that $F_3$ picks up two exceptional divisors.
Summarising, we see that $F_3$ gives rise to an $I_5$ fibre at $\lambda=3$. 
An Euler number count now shows that the other two singular fibres are of type $I_1$.

\begin{figure}[ht!]
	
\tikzset{every picture/.style={line width=0.5pt}} 

\begin{tikzpicture}[x=0.75pt,y=0.75pt,yscale=-1,xscale=1, scale=0.8]
	
	\draw    (238,56) -- (327.5,217) ;
	\draw    (500,144) -- (421.5,216) ;
	\draw    (497,69) -- (232.5,69) ;
	\draw [color={rgb, 255:red, 0; green, 97; blue, 208 }  ,draw opacity=1 ]   (370.5,52) .. controls (354.5,69) and (350,85) .. (378,99) ;
	\draw [color={rgb, 255:red, 0; green, 97; blue, 208 }  ,draw opacity=1 ]   (268.5,141) .. controls (282,127) and (327,133) .. (336,139) ;
	\draw [color={rgb, 255:red, 0; green, 97; blue, 208 }  ,draw opacity=1 ]   (425,109) .. controls (412,137) and (433.5,184) .. (462,201) ;
	\draw [color={rgb, 255:red, 208; green, 2; blue, 27 }  ,draw opacity=1 ]   (312,225) .. controls (319.5,195) and (418,189) .. (434,218) ;
	\draw    (479,60) -- (479,173) ;
	\draw [color={rgb, 255:red, 80; green, 227; blue, 194 }  ,draw opacity=1 ]   (365,177) .. controls (374,154) and (385,149) .. (414,154) ;
	\draw [color={rgb, 255:red, 80; green, 227; blue, 194 }  ,draw opacity=1 ]   (372,85) .. controls (370,121) and (388,150) .. (410,165) ;
	\draw [color={rgb, 255:red, 189; green, 16; blue, 224 }  ,draw opacity=1 ]   (301,113) .. controls (338,135) and (393,143) .. (433,113) ;
	\draw [color={rgb, 255:red, 0; green, 80; blue, 181 }  ,draw opacity=1 ]   (435,153) .. controls (458,127) and (474,122) .. (505,110) ;
	\draw [color={rgb, 255:red, 80; green, 227; blue, 194 }  ,draw opacity=1 ]   (309,107) .. controls (323,130) and (330,137) .. (327,166) ;
	\draw [color={rgb, 255:red, 189; green, 16; blue, 224 }  ,draw opacity=1 ]   (317,157) .. controls (363,145) and (387,161) .. (416,187) ;
	\draw [color={rgb, 255:red, 0; green, 80; blue, 181 }  ,draw opacity=1 ]   (418,171) -- (401,191) ;
	\draw [color={rgb, 255:red, 0; green, 80; blue, 181 }  ,draw opacity=1 ] [dash pattern={on 0.84pt off 2.51pt}]  (435,153) -- (418,171) ;
	
\end{tikzpicture}
\caption{\label{fig:7b} The surface $Y$ when $P = P_{7b}$.}
\end{figure}
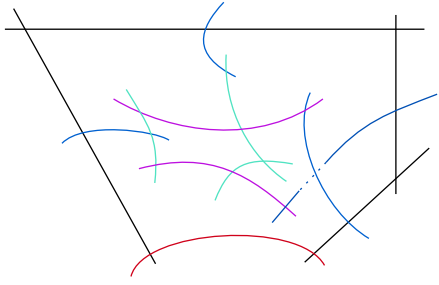

\subsection{8a}
Set $P = P_{8a}$.
Then $Y_P = \minresYP = \PP^1 \times \PP^1$ with coordinates $([x_0 : x_1], [y_0 : y_1])$ and
\[
F_\lambda = \{(x_0 + x_1)^2 (y_0 + y_1)^2 + (\lambda - 4) x_0 x_1 y_0 y_1 = 0\} \subset \PP^1 \times \PP^1.
\]
The base locus of $\pencilP$ consists of $4$ reduced points, namely $([1:-1],[0:1])$, $([1:-1],[1:0])$, $([0:1],[1:-1])$, $([1:0],[1:-1])$, and $4$ infinitely near base points of the first order.

The surface $Y$ is obtained by blowing up these $4$ points $2$ times.
The situation is described in Figure~\ref{fig:P8a}: the strict transforms of the $4$ components of the toric boundary of $\PP^1 \times \PP^1$ are depicted in black, are $(-2)$-curves, and constitute $D$; 
there are $8$ exceptional curves of $Y \to \PP^1 \times \PP^1$: $4$ of them, depicted in blue, are $(-1)$-curves; the remaining $4$, depicted in cyan, are $(-2)$-curves.

The reducible curve in purple in Figure~\ref{fig:P8a} is the strict transform of the curve $\{ (x_0 + x_1)(y_0 + y_1) = 0 \} \subset \PP^1 \times \PP^1$, which is the reduction of $F_4$.
The fibre $f^{-1}(4)$ is equal to the sum of: $2$ times the two purple curves, $1$ time the cyan curves.
Hence $f^{-1}(4)$ is of type $I_1^*$.

Considering the topological Euler characteristic we deduce that there must be also a $I_1$ fibre. Hence, the singular fibres are one $I_4$, one $I_1$ and one $I_1^*$.
\begin{figure}[ht!]
	\centering
	\includegraphics[width=0.25\textwidth]{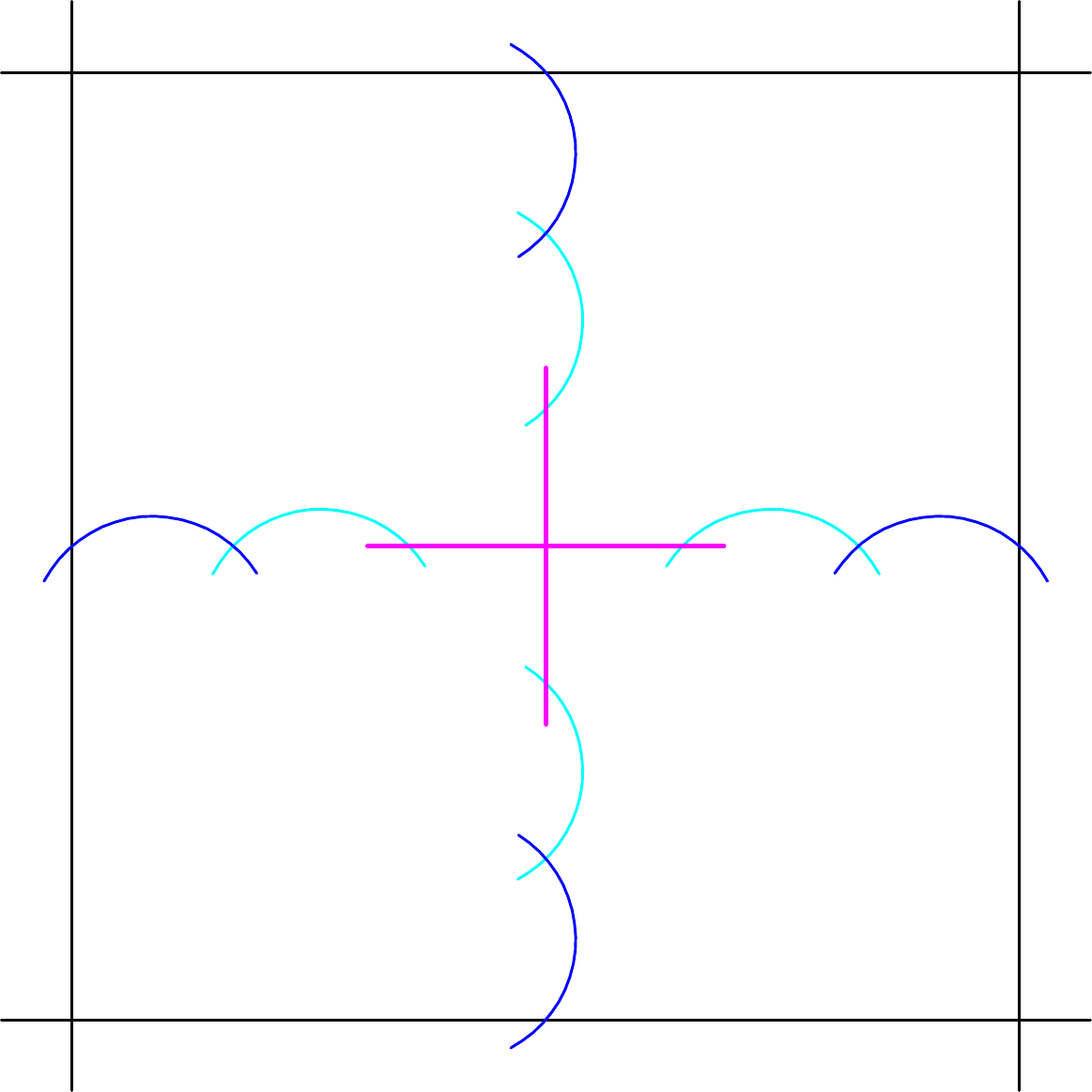}
	\caption{The surface $Y$ when $P = P_{8a}$.}
	\label{fig:P8a}
\end{figure}

\subsection{8b}
Let $P = P_{8b}$.
Then $Y_P = \minresYP$ is the first Hirzebruch surface $\FF_1$.
The fibre $D$ is of type $I_4$.

Let us use the affine chart $U$ of $Y_P$ isomorphic to $\AA^2$ associated to the cone with rays $e_1$ and  $e_2$.
Then $F_\lambda \cap U$ is isomorphic to
\[
C_\lambda = \left\{  (1+x)(1+y+xy)^2 + (\lambda - 4) xy = 0   \right\} \subset \AA^2.
\]

Clearly $C_4$ is reducible and has two components: $\{ 1 +x = 0\}$ with multiplicity $1$, and $\{ 1+xy+y  = 0 \}$ with multiplicity $2$.
These components do not intersect.
We obtain that $f^{-1}(4) \subset Y$ is an $I_1^*$ fibre, (the union of the cyan and purple curves in Figure~\ref{fig:P8b}).

With a topological Euler characteristic count, we see that there must be also an $I_1$ fibre.

\begin{figure}[ht!]
	\centering
	\includegraphics[width=0.4\textwidth]{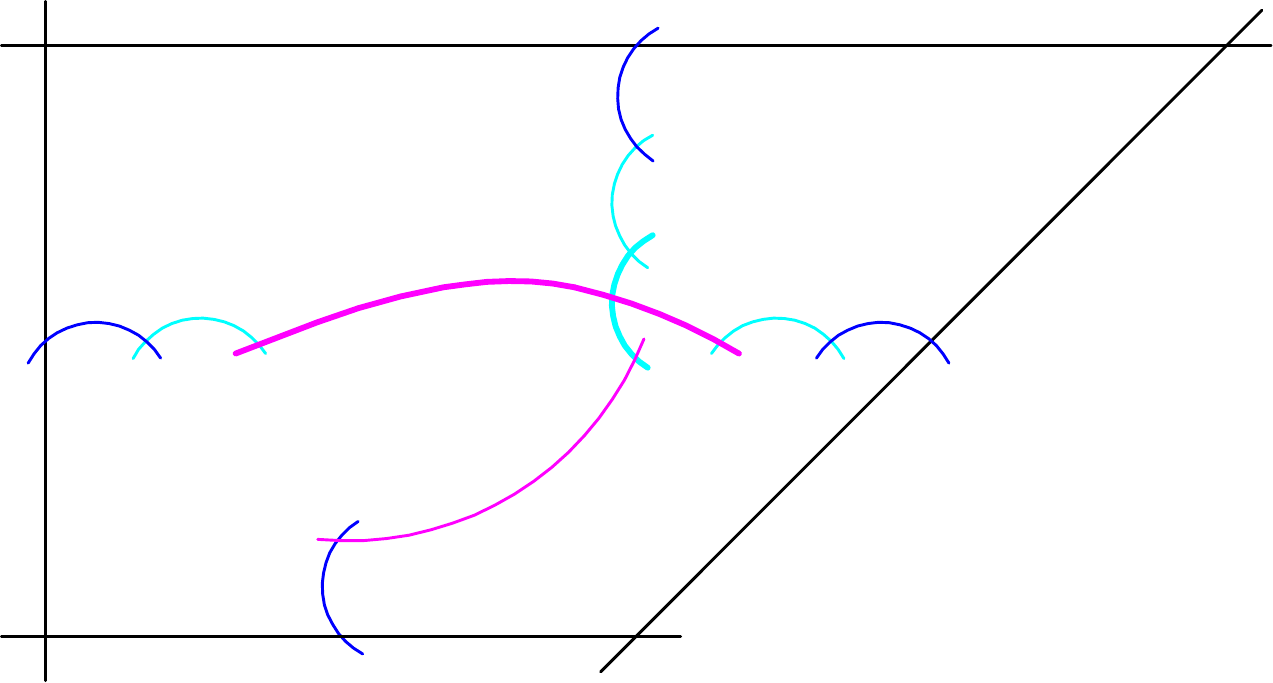}
	\caption{The surface $Y$ when $P = P_{8b}$.}
	\label{fig:P8b}
\end{figure}

\subsection{8c}

Let $P=P_{8c}$. The normal fan of $P$ is the face fan of $P_{4c}$. Then $Y_P$ is $\PP(1,1,2)$. It has one $A_1$ singularity. 
Let $x,y,z$ be weighted homogeneous coordinates on $\PP(1,1,2)$. 
Then $\textbf{1}_P$ is  $xyz$, and $f_P$ is 
$(x+y)^4+z(2x^2+2y^2+z)$.
The base locus of the pencil $\delta_P$ is made of $8$ points: 
$3$ reduced points and $3$ infinitely near base points of the first order at $[1:-1:0]$, $[0:1:-1]$, $[1:0:-1]$, and 
an infinitely near base pint of the second order and
an infinitely near base pint of the third order at $[1:-1:0]$.

The minimal resolution $\overline{Y}_P$ has one $(-2)$-curve and is the $2$nd Hirzebruch surface $\PP_{\PP^1}(\cO \oplus \cO(2))$. The strict transforms in $\overline{Y}_P$ of the components of $D_P$ are two $0$-curves and one $2$-curve. 

The exceptional curves of $Y \to  \overline{Y}_P$  are five $(-2)$-curves  (in cyan in Figure~\ref{fig:8c}) and three $(-1)$-curves  (in blue in Figure~\ref{fig:8c}). The divisor $D=f^{-1}(\infty)$ is a $I_4$ fibre.

The curve $F_4$ is the non-reduced curve:
\[
\left\{  \left((x+y)^2+z\right)^2=0\right\} \subset \PP(1,1,2).
 \]
The rational curve $\left\{ (x+y)^2+z=0 \right\}$ intersects $ \{x=0\} $ transversely at $[0:1:-1]$, $ \{y=0\}$ transversely at $[1:0:-1]$, and is tangent to $\{ z=0 \}$ at $[1:-1:0]$.
Then the fibre $f^{-1}(4)$ is given by the union of twice the strict transform of $\left\{ (x+y)^2+z=0 \right\}$ in $Y$ (in purple in Figure~\ref{fig:8c}), once the four $(-2)$-curves that do not intersect $\widetilde{F}_4$, and twice the $(-2)$-curve intersecting $\widetilde{F}_4$.
Thus $f^{-1}(4)$ is a fibre of type $I_1^\ast$.

The sum of the topological Euler characteristic of the singular fibres different from $I_4$ must be 8. 
It follows that $f$ has fibres of type $I_4, I_1^\ast, I_1$.

\begin{figure}[ht!]
	\tikzset{every picture/.style={line width=0.4pt}} 
	
	\begin{tikzpicture}[x=0.75pt,y=0.75pt,yscale=-1,xscale=1, scale=1.2]
		\path (230,180); 
		
		\draw    (393,49) -- (360.33,162.33) ;
		\draw    (237.67,52.33) -- (269.67,160.33) ;
		\draw    (233.67,59.67) -- (395.67,59.67) ;
		\draw [color={rgb, 255:red, 208; green, 27; blue, 2 }  ,draw opacity=1 ]   (252,159.25) .. controls (286.5,144.75) and (355,149.75) .. (377,159.25) ;
		\draw [color={rgb, 255:red, 2; green, 10; blue, 208 }  ,draw opacity=1 ]   (316,50.25) .. controls (314.5,53.75) and (307,76.25) .. (315.5,80.25) ;
		\draw [color={rgb, 255:red, 2; green, 10; blue, 208 }  ,draw opacity=1 ]   (242.5,123.5) .. controls (251.5,124.25) and (262,126.75) .. (270,112.25) ;
		\draw [color={rgb, 255:red, 80; green, 227; blue, 194 }  ,draw opacity=1 ]   (262,116.5) .. controls (272,117.25) and (283,122.25) .. (295.5,105.75) ;
		\draw [color={rgb, 255:red, 80; green, 227; blue, 194 }  ,draw opacity=1 ]   (314.5,72.25) .. controls (308.5,81.25) and (311.5,97.75) .. (315,102.75) ;
		\draw [color={rgb, 255:red, 80; green, 227; blue, 194 }  ,draw opacity=1 ]   (314.5,94.25) .. controls (311,95) and (307.5,112.25) .. (315,123.75) ;
		\draw [color={rgb, 255:red, 80; green, 227; blue, 194 }  ,draw opacity=1 ]   (316.5,117.75) .. controls (309.5,121.75) and (312,135.75) .. (316.5,138.75) ;
		\draw [color={rgb, 255:red, 2; green, 10; blue, 208 }  ,draw opacity=1 ]   (362.5,111) .. controls (368.5,122.75) and (379.5,118.25) .. (391,111.25) ;
		\draw [color={rgb, 255:red, 80; green, 227; blue, 194 }  ,draw opacity=1 ]   (338.5,105.25) .. controls (338.5,110.75) and (353,119.25) .. (367.5,111.75) ;
		\draw [color={rgb, 255:red, 208; green, 2; blue, 164 }  ,draw opacity=1, line width=1 ]   (283.5,104.75) .. controls (295,110.75) and (324.5,114.75) .. (349,105.75) ;
	\end{tikzpicture}
	\caption{\label{fig:8c} The surface $Y$ when $P=P_{8c}$.}	
\end{figure}
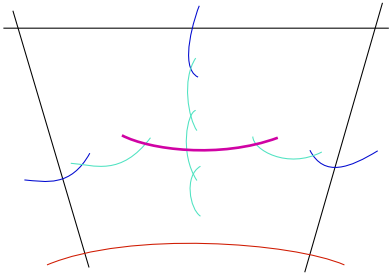

\subsection{9}
Set $P = P_9$.
Then $Y_P = \minresYP = \PP^2$ and
\[
F_\lambda = \{ (x_0 + x_1 + x_2)^3 + (\lambda -6) x_0 x_1 x_2 = 0 \} \subset \PP^2.
\]
The base locus of $\pencilP$ is made up of $9$ points: the $3$ points $[0:1:-1]$, $[1:0:-1]$, $[1:-1:0]$, $3$ infinitely near points of the first order,
$3$ infinitely near points of the second order.

The surface $Y$ is obtained by blowing up these $3$ points $3$ times. The situation is described in Figure~\ref{fig:P9}: the strict transforms of the coordinate lines of $\PP^2$ are depicted in black and are $(-2)$-curves;
there are $9$ exceptional curves of $Y \to \PP^2$: $3$ of them, depicted in blue, are $(-1)$-curves; the remaining $6$, depicted in cyan and in green, are $(-2)$-curves.

The curve depicted in purple in Figure~\ref{fig:P9} is an interesting curve: it is the strict transform of the line $\{ x_1 + x_2 + x_3 = 0 \} \subset \PP^2$, which is the reduction of $F_6$.
One sees that the fibre $f^{-1}(6)$ is equal to the sum of: $3$ times the purple curve, $2$ times the green curves, $1$ time the cyan curves. Hence $f^{-1}(6)$ is of type $IV^*$.

Considering the topological Euler characteristic we deduce that there must be also a $I_1$ fibre. Hence, the singular fibres are one $I_3$, one $I_1$ and one $IV^*$.

\begin{figure}[ht!]
	\centering
	\includegraphics[width=0.35\textwidth]{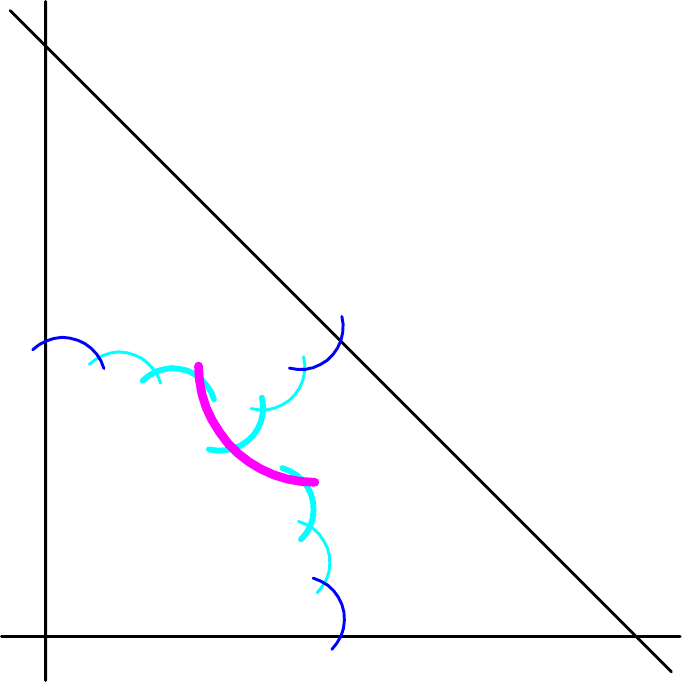}
	\caption{The surface $Y$ when $P = P_9$.}
	\label{fig:P9}
\end{figure}

\subsection{Conclusion} \label{sec:conclusion}

The analysis of the singular fibres of the elliptic fibrations $Y \to \PP^1$ constructed as in Construction~\ref{constr:2} from all reflexive polygons $P$ is summarised in Table~\ref{tab:summary_singular_fibre}. 
The properties of the sections described in the examples can be calculated as we do for the polytopes 3, 4a, and 4b.
In the same table we also list the corresponding number in \cite[Table~8.2]{schuett_shioda_book} and the corresponding Mordell--Weil group.

\begin{table}[ht!]
	\begin{equation*}
		\begin{array}{llcc}
			P & \text{Singular fibres of } f  & \text{No.} & \MW(Y)  \\
			\hline
			P_3 & I_9, \ 3 \times I_1  & 63 & \ZZ/3\ZZ   \\
			P_{4a}, P_{4c} & I_8, \ I_2, \ 2 \times I_1 & 70 & \ZZ / 4 \ZZ   \\
			P_{4b} & I_8, \ 4 \times I_1 & 45 & \ZZ  \\
			P_{5a}, P_{5b} & I_7, \ I_2, \ 3 \times I_1 & 47 & \ZZ \\

			P_{6a},
			P_{6b},
			P_{6c},
			P_{6d} &  I_6, \ I_3, \ I_2, \ I_1  & 66 & \ZZ / 6\ZZ \\
			P_{7a},
			P_{7b} & 2 \times I_5, \ 2 \times I_1 & 67 & \ZZ / 5 \ZZ \\ 
			P_{8a},
			P_{8b},
			P_{8c} & I_4, \ I_1^*, \ I_1  & 72 & \ZZ / 4 \ZZ   \\
			P_9        & I_3, \ IV^*, \ I_1 & 69 & \ZZ/3\ZZ  \\
		\end{array}
	\end{equation*}
	\caption{ 	\label{tab:summary_singular_fibre} The singular fibres, the number in \cite[Table~8.2]{schuett_shioda_book}, and the Mordell--Weil group of the rational elliptic surface $Y$ constructed as in Construction~\ref{constr:2}, for each reflexive polygon $P$.
	Different reflexive polygons can give the same rational elliptic surface (see Remark~\ref{rmk:different_polygons_can_give_the_same_surface}).}
\end{table}

We have the following consequences:
\begin{itemize}
	\item if $P$ is not $\GL_2(\ZZ)$-equivalent to $P_{4b}$ nor to $P_{5a}$ nor to $P_{5b}$, then $Y$ is extremal and $\MW(Y)$ is a finite cyclic group;
	\item if $P$ is $\GL_2(\ZZ)$-equivalent to $P_{4b}$ or to $P_{5a}$ or to $P_{5b}$, then $\MW(Y)$ is an infinite cyclic group.
\end{itemize}

\begin{remark} \label{rmk:different_polygons_can_give_the_same_surface}
	Another interesting feature is that, with only one exception, reflexive polygons with the same volume give the same rational elliptic surface. The exception is for reflexive polygons with volume $4$: the rational elliptic surface of $P_{4b}$ is different from the rational elliptic surface of $P_{4a}$ and of $P_{4c}$.
	We explain this phenomenon below.
\end{remark}

\begin{remark}
	The classification of the singular fibres of the elliptic fibrations associated to $P_3$ and to $P_{4a}$ also appears in \cite[p.~504]{doran_kerr}.
\end{remark}


\section{Mutations}

\subsection{Algebraic mutations} \label{sec:algebraic_mutations}

Let $M$ be a lattice of rank $n$ with dual lattice $N=\Hom_\ZZ(M,\ZZ)$ and consider the algebraic torus $T_M= \Spec \CC[N] = M \otimes_\ZZ \CC^\times$. 
Let $v \in M$ be primitive and $h \in \CC[v^\perp] \subset \CC[N]$.
Following \cite{fomin_zelevinsky, GrossHackingKeelClusterAlgebras, ghkk}, we define the automorphism of the function field $\CC(N) = \mathrm{Frac} \ \CC[N]$
\begin{equation*}\label{eq:AlgebraicMutation}
	x^{u} \mapsto x^{u} h^{-\langle {u}, {v} \rangle}
\end{equation*}
which induces a birational map 
\[ \mu_h \colon T_M \dashrightarrow T_M. \]
We call $\mu_h$ an \emph{algebraic mutation}, and $h$ the factor of the mutation.
If we extend $v$ to a basis $e_1=v, e_2, \dots e_n$ of $M$ and $x_1, \dots, x_n$ are the coordinates on $T_M$ which correspond to the dual basis $e_1^*,\dots, e_n^*$ of $N$, then $h$ is a Laurent polynomial in $x_2, \dots x_n$ and $\mu_h$ is given by 
\begin{equation*}
	(x_1, \dots x_n) \mapsto (h(x_2, \dots x_n)^{-1}x_1, x_2, \dots x_n).
\end{equation*}
Let $\PP$ be the toric variety defined by the fan consisting of the two rays $\RR_{\geq 0}v$ and $\RR_{\leq 0}v$ in the lattice $M$.
$\PP$ is isomorphic to $\PP^1 \times T_{M/\ZZ v}$, and the projection to $\PP^1$ is induced by the lattice homomorphism $M \onto M/{\ZZ v}$. $\PP$ comes with two toric divisors $D_+$ and $D_-$.
Since $\Hom_\ZZ (M/{\ZZ v}, \ZZ)=v^\perp$, $h$ is canonically a regular function on the torus $T_{M/{\ZZ v}}$, and we write $Z_\pm=\pi^{-1}(\rV(h)) \cap D_{\pm} \subset \PP$, where $\pi \colon \PP = \PP^1 \times T_{M / \ZZ v} \to T_{M / \ZZ v}$ is the second projection.  
Let $b_{\pm} \colon \tilde{\PP}_{\pm} \rightarrow \PP$ be the blowup of $\PP$ at $Z_{\pm}$.

\begin{lemma}[{\cite[Lemma~3.2]{GrossHackingKeelClusterAlgebras}}]\label{lem:mutationextends}	$\mu_h$ extends to a regular isomorphism $\tilde{\PP}_+ \rightarrow~ \tilde{\PP}_-.$
\end{lemma}

\[
\xymatrix{
	&  \tilde{\PP}_+ \ar[d]_{b_+} \ar[r]^{\simeq} & \tilde{\PP}_- \ar[d]^{b_-} \\
	Z_+ \ar@{^{(}->}[r] & \PP \ar@{-->}[r] & \PP & Z_- \ar@{_{(}->}[l] \\
	& T_M \ar@{^{(}->}[u] \ar@{-->}[r]^{\mu_h} & T_M \ar@{^{(}->}[u]
}
\]

\begin{proof}
	Note first that $\mu_h$ extends to a birational map on $\PP=\PP^1_{x_1,y_1} \times T_{M/\ZZ v}$ given by	\begin{equation*}
		\mu_h \colon ([x_1 \colon y_1], x_2, \dots x_n)\mapsto 	([x_1 \colon h(x_2, \dots x_n)y_1], x_2, \dots x_n)
	\end{equation*}
	$\mu_h$ is undefined iff $x_1=0$ and $h=0$, i.e.\ exactly at $Z_+$, and $\mu_h^{-1}$ is undefined where $y_1=0$ and $h=0$, i.e.\ at $Z_-$. 
	By definition, $\tilde{\PP}_\pm$ are the subvarieties of $\PP^1_{s,t} \times \PP^1_{x_1,y_1} \times T_{M/\ZZ v}$ cut out by the equations
	\[
	x_1t-y_1sh=0, \quad y_1t-x_1sh=0
	\]
	respectively. Noting that $[s:t]=[x_1:hy_1]$ for $([s:t],[x_1:y_1], x_2, \dots x_n) \in \tilde{\PP}_+$ away from the exceptional divisor, it follows that the isomorphism $\tilde{\PP}_+ \rightarrow \tilde{\PP}_-$ defined by
	\[
	([s:t], [x_1 \colon y_1], x_2, \dots x_n) \mapsto 	([y_1:x_1], [s:t], x_2, \dots x_n)
	\]
	gives the required extension of $\mu_h$. 
\end{proof}
In other words, $\mu_h$ is the map which blows up $Z_+$ and blows down the strict transform of the fibre through $Z_-$. 

\subsection{Combinatorial mutations of reflexive polygons}

Let us now specialise to the situation at hand, where $N$ and $M$ have rank $2$.
We make the following definition, which is a special case of \cite[Definition~5]{sigma}.
\begin{definition}\label{mutationdef}
	Let $P$ be a reflexive polygon in the rank $2$ lattice $N$ and let $v \in M$ be the inner normal to an edge of $P$. Choose a primitive line segment $H \subset v^{\perp} \subset N$.
	For every $d \in \ZZ$, write $P_d$ 
	for the slice of $P$ at height $d$ with respect to $v$, i.e.\
	\[
	P_d = \{ x \in P \mid \langle v, x \rangle = d \}.
	\]
	We assume that $P_d$ is empty for $d>1$ (but see Remark~\ref{rmk:mutation} below).
	Decompose $P_{-1}=R_{-1}+H$ as a Minkowski sum for 
	some line segment $R_{-1}$.
	Then the \emph{combinatorial mutation} of $P$ with respect to $(v, H)$ is defined to be the reflexive polygon
	\[
	P^\dagger = \text{conv} \left( R_{-1} \cup P_0 \cup (P_1+H) \right).
	\]
	
\end{definition}
\begin{remark}\label{rmk:mutation}
	This definition is usually stated in more generality; for instance see \cite{sigma}. In particular, we usually do not require $H$ to be primitive, and the condition that $P_d$ be empty for $d>1$ is only needed to ensure that $P^\dagger$ is again a reflexive polygon. If one works with the larger class of Fano polygons, this condition can and should be removed.
\end{remark}
\begin{figure}
	\centering
	\begin{tikzpicture}[scale=1]
		\begin{scope}[blend mode=multiply]
			\draw (0,-1) -- (1,1) -- (-1,1) --(0,-1);
			\node  at (-1,-1) {$\cdot$};
			\node at (0,-1) {$\cdot$};
			\node at (1,-1) {$\cdot$};
			\node  at (-1,0) {$\cdot$};
			\node  at (0,0) {$\times$};
			\node  at (1,0) {$\cdot$};
			\node at (-1,1) {$\cdot$};
			\node  at (0,1) {$\cdot$};
			\node  at (1,1) {$\cdot$};
		\end{scope}
		\begin{scope}[xshift=7cm]
			\draw (0,-1) -- (1,-1) -- (0,1) --(-1,1)--(0,-1);
			\node  at (-1,-1) {$\cdot$};
			\node at (0,-1) {$\cdot$};
			\node at (1,-1) {$\cdot$};
			\node  at (-1,0) {$\cdot$};
			\node  at (0,0) {$\times$};
			\node  at (1,0) {$\cdot$};
			\node at (-1,1) {$\cdot$};
			\node  at (0,1) {$\cdot$};
			\node  at (1,1) {$\cdot$};
		\end{scope}
		\draw[->] (2.3,0) -- (4.8,0)node[midway,fill=white]{$\mathrm{mut}_{F,w}$};
	\end{tikzpicture}
	\caption{Mutation of the polygon $P_{4c}$ with respect to mutation data given by $v=(0,-1)$ and $H=\Newt{1+x}$.
		The mutated polygon is $\SL_2(\ZZ)$-equivalent to $P_{4a}$.}
	\label{fig:combinatorialmutation}
\end{figure}
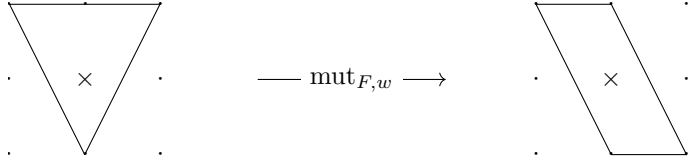
The effect of a mutation on the normal fan $\Sigma_P$ of $P$ is easy to describe.
Note that $\Sigma_P$ contains the ray $\RR_{\geq 0}v$, whereas $\Sigma_{P^\dagger}$, the normal fan of the mutated polygon $P^\dagger$, contains the ray $\RR_{\leq 0}v$.
Let $\Sigma_+$ (resp.\ $\Sigma_-$) be the fan obtained by adding the ray $\RR_{\leq 0}v$ (resp.\ $\RR_{\geq 0} v$) to $\Sigma_P$ (resp.\ $\Sigma_{P^\dagger}$).
Assume $H = \mathrm{conv}(0,w)$, where $w \in N$ is primitive.
Define the piecewise linear map
\[
\trop \colon M_\RR \rightarrow M_\RR, \quad m \mapsto m-\min\{0, \langle m,w \rangle \} v
\]
$\trop$ acts as the identity on the half space $\langle \cdot , w \rangle >0$ and acts as a simple shear on the half space $\langle \cdot , w \rangle <0$.
Then we see that $\Sigma_-$ is obtained by applying $\trop$ to each ray of $\Sigma_+$.
Let $D_+$ be the toric divisor on $Y_P$ corresponding to $\RR_{\geq 0}v$ in $\Sigma_+$ and $D_-$ be the toric divisor on $Y_{P^\dagger}$ corresponding to $\RR_{\leq 0}v$ in $\Sigma_{-}$. Let $v_0=~v,v_1, \dots v_m$ be the inner normals to the edges of $P$. Define $D_{j,+}$ to be the divisor corresponding to $\RR_{\geq 0}v_j$ in $\Sigma_+$ and $D_{j,-}$ be the divisor corresponding to $\RR_{\geq 0}\trop(v_j)$ in $\Sigma_{-}$. In particular, we have that $D_+=D_{0,+}=D_{0,-}$. 
Given $h \in \CC[N]$, and a fan $\Sigma  \subset M$, we write $\bar{V}(h)$ for the closure of $V(h) \subset T_M$ in $Y_\Sigma$. 

Let $h=1+x^w$ (note that $\Newt{h}=H$) and define $Z_{\pm}=\bar{V}(h) \cap D_{\pm}$.
Let $w_0=w$ (recall that $w \in v^\perp$), and for each other $j$, choose a primitive generator $w_j$ of $v_j^\perp$ and set 
\[
h_j=\begin{cases}
	(1+x^{w_j})^{\ell_j} \quad \text{for} \; j \neq 0\\
	(1+x^{w_j})^{\ell_j-1} \quad \text{for} \; j = 0\\
\end{cases}
\]
where $\ell_j$ the lattice length of the edge corresponding to $v_j$. 
Following \cite[Lemma 3.6]{GrossHackingKeelClusterAlgebras}, we define 
\begin{align*}
	Z_{j,+}&=\bar{V}(h_j) \cap D_{j,+} \\
	Z_{j,-}&=
	\begin{cases}
		\bar{V}(h_j) \cap D_{j,-} &\quad \text{if} \;\langle w,v_j \rangle\geq 0\\
		\bar{V}((1+x^{w_j+\langle w_j,v \rangle w})^{\ell_j}) \cap D_{j,-} &\quad \text{if}\; \langle w,v_j \rangle< 0
	\end{cases}
\end{align*}
 
 Note that the divisor given by the sum of the points $Z_+$ and the $Z_{j_+}$ on the toric boundary of $Y_{\Sigma_+}$
is by construction the base locus of the pencil of $f_P$. 
Note also that $D_{j,-}$ has inner normal $v_j'=v_j-\langle v_j, w \rangle v$ so that a primitive generator for $v_j'^\perp$ is given by $w_j+\langle w_j,v \rangle w$. It follows that the divisor given by the sum of $Z_-$ and the $Z_{j_-}$ on the toric boundary of $Y_{\Sigma_-}$
is the base locus of the pencil of $f_{P^\dagger}$.

We have the following result, which is a strengthening of \cite[Lemma 3.6]{GrossHackingKeelClusterAlgebras} for the very special situation at hand. We will closely follow their proof, adapting it to our notation.

\begin{prop} \label{prop:mutation_eq_polygons_give_the_same_res}
	Let $v\in M$, let $w \in N$,  let $H=\mathrm{conv}(0,w)$ and let $h, Z_{\pm}, Z_{j, \pm}$ as above. Suppose that $P$ and $P^\dagger$ are reflexive polygons such that $P^\dagger$ is obtained from $P$ by a mutation with respect to $(v, H)$. 
	Let $Y$ (resp.\ $Y^\dagger$) be the rational elliptic surface obtained from the polygon $P$ (resp.\ $P^\dagger$) as in Construction~\ref{constr:2}.  
	Then $Y$ and $Y^\dagger$ are isomorphic. 
\end{prop}

The idea of the proof is that algebraic mutations (for Laurent polynomials in $2$ variables) and combinatorial mutations (for polygons) are actually the same thing. We try to informally explain this now. If $P$ and $P^\dagger$ are related via a combinatorial mutation, then the pencils $\pencilP$ and $\mathfrak{d}_{P^\dagger}$ are closely related. Indeed, if we consider the sections $f_P$ and $f_{P^\dagger}$ of the anticanonical line bundles of toric surfaces $Y_P$ and $Y_{P^\dagger}$ as in Construction~\ref{constr:1} and we consider their restrictions to the tori $T_M \subset Y_P$ and $T_M \subset Y_{P^\dagger}$, then these restrictions are regular functions on $T_M$ (i.e.\ elements of the ring $\CC[N]$) and they are related via an algebraic mutation $T_M \dashrightarrow T_M$. The reason is that if one applies the functor ``Newton polytope'' to an algebraic mutations between Laurent polynomials then one gets a combinatorial mutation between their Newton polytopes. This birational selfmap of the torus $T_M$ gives rise to an isomorphism between $Y$ and $Y^\dagger$ (which are the rational elliptic surfaces associated to $P$ and to $P^\dagger$, respectively).

\begin{proof}[Proof of Proposition~\ref{prop:mutation_eq_polygons_give_the_same_res}]
	We first show that the mutation $\mu_h \colon T_M \dashrightarrow T_M$ extends to a regular isomorphism after only blowing up $Z_{+}$ on $Y_{\Sigma_+}$ and $Z_{-}$ on $Y_{\Sigma_-}$, and then show that $\mu_h(Z_{j,+})=Z_{j,-}$. This will show that $\mu_h$ extends to an isomorphism after blowing up the base locus of $f_P$ on $Y_{\Sigma_+}$ and  the base locus of $f_{P^\dagger}$ on $Y_{\Sigma_-}$. By definition of $Y$ and $Y^\dagger$, this will then give the required result.
	
	Abusing notation, let temporarily $Y$ be the blowup of $Y_{\Sigma_+}$ along $Z_+$ and $Y^\dagger$ the blowup of $Y_{\Sigma_-}$ along $Z_-$. 
	Let $U \subset Y$ be the union of $\tilde{\PP}_+$ for $v=v_0$ and the open subsets of the form $U_\rho \setminus \bar{V}(h)$ where $\rho$ ranges over rays of $\Sigma_P$ not equal to $\RR_{\geq 0}v$ or $\RR_{\leq 0}v$, and $U_\rho$ is the affine toric variety associated to the fan with only one ray $\rho$, i.e.\ the union of the dense torus and the toric divisor $D_\rho$. 
	We claim that in our situation, these open sets actually cover $Y$. Indeed, note that $D_\rho \cap \bar{V}(h)=\emptyset$ if $\langle w, \rho \rangle \neq 0$, since then either $x^{w}$ or $x^{-w}$ vanishes along $D_\rho$ and therefore $\bar{V}(1+x^{w})=\bar{V}(1+x^{-w})$. So we only fail to cover codimension $2$ sets of the form $D_\rho \cap \bar{V}(h)$ such that $w$ is zero on $\rho$. Since we are in dimension $2$, this can only happen if $\rho=\RR_{\geq 0}v$ or $\RR_{\leq 0}v$, so by definition of the range of $\rho$, there are no such sets, and therefore $U=Y$.
	
	By Lemma~\ref{lem:mutationextends}, $\mu_h$ extends to a well-defined morphism on the open set isomorphic to $\tilde{\PP}_+$, so we need to check that $\mu_h$ is well-defined on the remaining sets.
	If $\langle w, \rho \rangle>0$, then $h_j \equiv 1$ on $D_\rho$. For any $n \in \rho^\vee \cap N=\trop(\rho)^\vee \cap N$, we have that  
	\[
	\mu_h^*(x^n)=x^nh^{-\langle n, v \rangle}
	\]
	so it follows that $\mu_h$ takes regular functions to regular functions on $U_\rho \setminus \bar{V}(h)$. 
	If $\langle w, \rho \rangle<0$, then $h_j$ is not defined on $D_\rho$. 
	For any $n \in \trop(\rho)^\vee \cap N$, we have that  
	\[
	\mu_h^*(x^n)=x^n(1+x^{w})^{-\langle n, v \rangle}=x^{n-\langle n, v \rangle w}(1+x^{-w})^{-\langle n, v \rangle}
	\]
	and $n-\langle n, v \rangle w \in \rho^\vee$ by definition of $\trop$, so that this is again a regular function on $U_\rho \setminus \bar{V}(h)$. This shows that $\mu_h$ extends to a regular morphism on $Y$, and we can repeat the same argument for $\mu_h^{-1}$ to show that $\mu_h$ defines an isomorphism $Y \rightarrow Y^\dagger$.
	
	To complete the proof, it suffices to show that $\mu_h(Z_{j,+})=Z_{j,-}$.
	We work  by cases again. If $\langle w, v_j \rangle > 0$, then $h|_{D_j} \equiv 1$, so
	$\mu_h^*(f_j)|_{D_j}=f_j|_{D_j}$. If $\langle w, v_j \rangle=0$, then we must have $\langle w_j, v \rangle=0$ as well, so that $\mu_h^*x^{w_j}=x^{w_j}$ and hence also $\mu_h^*f_j=f_j$.  
	If $\langle w, v_j \rangle \leq 0$, then (noting the definition of $Z_{j,-}$ in this case)
	\begin{align*}
		\mu_h^*((1+x^{w_j+\langle w_j,v \rangle w})^{\ell_j})&=(1+x^{w_j+\langle w_j,v \rangle w}(1+x^w)^{-\langle w_j, v \rangle})^{\ell_j}\\
		&=(1+x^{w_j}(1+x^{-w})^{-\langle w_j, v \rangle})^{\ell_j}
	\end{align*}
	Since $x^{-w}$ vanishes along $D_j$ in this case, we obtain that 
	\[
	\mu_h^*((1+x^{w_j+\langle w_j,v \rangle w})^{\ell_j})|_{D_j}=(1+x^{w_j})^{\ell_j}|_{D_j}
	\]
	and therefore that $\mu_h(Z_{j,+})=Z_{j,-}$, as required.
\end{proof}

Proposition~\ref{prop:mutation_eq_polygons_give_the_same_res} explains why, with only one exception, two reflexive polygons with the same volume give the same rational elliptic surface (see Remark~\ref{rmk:different_polygons_can_give_the_same_surface}).
Indeed, from Figure~\ref{fig:combinatorialmutation}, we have that $P_{4c}$ and $P_{4a}$ are mutation equivalent, hence the corresponding rational elliptic surfaces are isomorphic. However, one could prove that $P_{4a}$ is not mutation equivalent to $P_{4b}$, and indeed their corresponding rational elliptic surfaces are different.
More generally, it is easy to verify that if $P$ and $P^\dagger$ are reflexive polygons with the same volume and their volume is different from $4$, then $P$ and $P^\dagger$ are mutation equivalent.
This explains why there are so few rational elliptic surfaces in Table~\ref{tab:summary_singular_fibre}.

\section{Periods of Laurent polynomials}
In this section we recall the notion of classical period of a Laurent polynomial\footnote{
The notion of classical period we present here is the one that appears in the Mirror Symmetry program~\cite{CCGG}. 
It also arises in more recent developments on Mirror Symmetry by Mandel \cite[\S 1.4]{mandel_fano_mirror}.}
 and we describe the local systems encoding the variation of cohomology of the elliptic fibrations $f \colon Y \to \PP^1$ studied in \S\ref{sec:analysis}.



\subsection{The classical period of a Laurent polynomial}
Let $(\CC^\times)^n$ be an algebraic torus with coordinates $x_1, \dots, x_n$.

\begin{definition}[{\cite[Definition 3.1]{CCGG}}] \label{eq:def-classical-period}
	Let $g \colon (\CC^\times)^n \to \CC$ be a Laurent polynomial, i.e.\ an element of the ring $ \CC[x_1^{\pm 1} , \dots, x_n^{\pm 1}]$. 
	The \emph{classical period} of $g$ is defined as:
	\begin{equation}
		\pi_g(t)=\int_{\Gamma}	\frac{1}{1-t g} \Omega
	\end{equation}
	where
	\[
	\Omega=\left(\frac{1}{2\pi \mathrm{i}} \right)^n \frac{\rmd x_1 \wedge \cdots \wedge \rmd x_n}
	{x_1 \cdots x_n}
	\]
	 is the normalised holomorphic volume form on $(\CC^\times)^n$, and 
	$\Gamma=\left( |x_1|=\dots = |x_n|=1
	\right)  \subset (\CC^\times)^n$ is the real compact torus.
	For $|t|$ very small, we have $\Gamma \subset 
	(\CC^\times)^n \setminus (1-tg=0)$, thus the integral is well defined. 
\end{definition}

The period $\pi_g(t)$ is solution to a differential operator  $L \in \CC\langle t, D\rangle$, where $D=t \frac{\rmd}{\rmd t}$. 
To see this, one can use the fact that
$\pi_g$ is a specialisation of certain 
solutions to Gel'fand--Kapranov--Zelevinsky (GKZ) hypergeometric systems, see \cite[Theorem 3.2]{CCGG}. 

\begin{definition}[{\cite[Definition 3.3]{CCGG}}]
	Write $L \in  \CC\langle t, D\rangle$ as  $L=\sum_{k=0}^h p_k(t) D^k$, with $p_k \in \CC[t]$.
	The \emph{Picard--Fuchs operator} $L_g$  of a Laurent polynomial $g$ is the unique operator  (up to multiplication by  a constant)
	such that $L_g \cdot \pi_g=0$, the integer $h$ is as small as possible and, once $h$ is fixed, the degree $\deg p_k$ is  as small as possible. We refer to the integer $h$ as the \emph{order} of $L_g$. Note that $L_g$ only depends on $\pi_g$, that is, $L_g=L_{g^\dagger}$ if $\pi_g=\pi_{g^\dagger}$.
\end{definition} 

The local system of solutions of $L_g$ is a complex local system of rank $h$ on $U=\PP^1 \setminus S$, where $S$ is the set of singularities of $L_g$.\footnote{The operator $L_g$ can have apparent singularities, i.e.\ singularities around which the monodromy representation is trivial. Here $S$ is the set of genuine singularities of $L_g$.} We denote it by $\mathrm{Sol}(L_g)$.

\begin{remark} \label{rem:recursion}
	Write $L \in  \CC\langle t, D\rangle$ as  $L=\sum_{j=0}^{l} t^j P_j(D)$, with	$P_j(D) \in \CC[D]$. 
	A formal series $\sum_{m=0}^\infty c_m t^m $ is annihilated by the differential operator $L$ if and only if 
	it satisfies the \emph{linear recursion relation}:
	\[
	\sum_{j\leq m} P_j(m-j)c_{m-j}=0 \ \forall m \geq 0.
	\]
	By expanding $\frac{1}{1-tg}$ in power series and applying iteratively the Residue Theorem, one finds that around $t=0$ the classical period $\pi_g$ is defined by the power series:
	\begin{equation}\label{eq:period-expansion}
		\pi_{g}(t)= \sum_{m=0}^\infty c_\mathrm{1}(g^m) t^m \in \CC \pow{t}
	\end{equation}
where $c_\mathrm{1}$ denotes the coefficient of the monomial $\mathrm{1}$.
	Then, one can compute the Picard--Fuchs operator $L_g$ by calculating enough coefficients $c_\mathrm{1}(g^m)$ of the series \eqref{eq:period-expansion} to guess the linear recursion relation.
\end{remark}

\begin{remark} \label{rem:period-invariance}
	Let $g$ be a Laurent polynomial, let $\mu \colon (\CC^\times)^n \dashrightarrow (\CC^\times)^n$ be a volume-preserving birational map, and let $g^\dagger = g \circ \mu$. In general $g^\dagger$  is not a Laurent polynomial but only a rational function. Despite this, it still makes sense to define the period $\pi_{g^\dagger}$ as above.
	Then, an application of the change-of-variables formula to \eqref{eq:def-classical-period} gives that 
	$\pi_{g}=\pi_{g^\dagger}$.
	Observe that any algebraic mutation, introduced in \S\ref{sec:algebraic_mutations}, is volume preserving. Therefore Laurent polynomials that are mutation-equivalent have the same classical period.
\end{remark}

\subsection{Certain Laurent polynomials on reflexive polygons and their periods} \label{sec:Laurent-fP}

Fix a reflexive polygon $P$ in a rank $2$ lattice $N$.
In Construction~\ref{constr:1} we have introduced the $T_M$-toric surface $Y_P$ and a specific section $f_P$ of a line bundle $L_P$ on $Y_P$.
This line bundle is canonically trivial away from the toric boundary $D_P \subset Y_P$ because $D_P$ is in the linear system $| L_P |$, therefore the restriction of $f_P$ to $T_M = Y_P \setminus D_P$ can be identified with a regular function on the torus $T_M = \Spec \CC[N]$, i.e.\ with a Laurent polynomial in $\CC[N]$.
With small abuse of notation, we use the symbol $f_P$ also to denote  this Laurent polynomial.
\[
\xymatrix{
& Y \ar[d] \ar@/^2pc/[dd]^f \\
T_M \ar@{^{(}->}[r] \ar[d]_{f_P} \ar@{^{(}->}[ru] &Y_P \ar@{-->}[d]     \\
\AA^1 \ar@{^{(}->}[r] & \PP^1
}
\]
In other words, $f_P$ is the Laurent polynomial in $\CC[N]$ with zero coefficient for the constant monomial, and binomial coefficients for the monomials corresponding to the edges of $P$. 
We can consider the classical period $\pi_{f_P}$.

The curve $(1-t f_P=0) \subset (\CC^\times)^2$ along which the integrand form in the definition of  $\pi_{f_P}(t)$ (see \eqref{eq:def-classical-period}) is meromorphic in the intersection\footnote{
The minus sign in front of $-1/t$ is due to our conventions at the beginning of \S\ref{sec:analysis}.}
\[ F_{-1/t} \cap T_M    \simeq 
f^{-1}(-1/t) \cap T_M.  \] 
The Picard--Fuchs operator $L_{f_P}$ is an irreducible order-two differential operator, thus the local system $\mathrm{Sol}(L_{f_P})$ is an irreducible rank-two local system on the complement $U=\PP^1 \setminus S$ of the singularities of  $L_{f_P}$. 

It follows that we have the identities:
\begin{equation} \label{eq:local-systems}
	\mathrm{Sol}(L_{f_P})=\mathrm{gr}_1^W R^1 ({f_P}_{U })_{!}\; \ZZ = 
	R^1 {(f_{U})}_{ !} \; \ZZ
\end{equation}
where ${f_P}_U$ is the restriction of $f_P \colon (\CC^\times)^2 \to \CC$ to the preimage of $U$ via $f_P$ and $f_U$ is the restriction of $f \colon Y \to \PP^1$ to the preimage of $U$ via $f$.
Indeed, on the one hand, $\mathrm{Sol}(L_{f_P})$ is an irreducible summand of $\mathrm{gr}_1^W R^1 ({f_P}_{U })_{!} \ZZ$
(see \cite[Remark 3.4]{CCGG}), on the other hand, $\mathrm{gr}_1^W R^1 ({f_P}_{U })_{!} \ZZ= R^1 {(f_{U})}_{ !}  \ZZ
$ has rank two since $f$ is an elliptic fibration.

\begin{example} We continue our running Example~\ref{ex:cubic_2}, so $P=P_3$. Then $\pi_{f_P}$ is the series:	 
	\[ \pi_{f_P}(t)=\sum_{j=0}^{\infty} \frac{(3j)!}{(j)!^3} t^{3j}
	\] This series satisfies the two-term recursion: 
	\[  j^2 c_{3j} - 3(3j-1)(3j-2) c_{3j-3}=0  \quad \forall j \geq 1
	\] 
	Setting $P_0(3j)=j^2$ we have that $P_0(0) \cdot c_0=0$. 
	Then by Remark \ref{rem:recursion} the Picard--Fuchs operator $L_{f_P}$ is the irreducible order-two operator:
	\[ L_{f_P}=\frac{1}{27} \cdot  D^2 - t^3 (D+2)(D+1)
	\]
	Note that $L_{f_P}$ is singular at $t=0, \frac{1}{3}, \zeta \cdot \frac{1}{3}, \zeta^2 \cdot \frac{1}{3}, \infty$, with $\zeta$ a primitive third root of unity. The point $t=\infty$ is an apparent singularity of $L_{f_P}$.
	This is consistent with our analysis of the local system $R^1 f_! \ZZ$ (i.e.\ of the singular  fibres and of the monodromy) of the family of curves $f \colon Y \to \PP^1$ in \S\ref{sec:3}.
\end{example}

\section{Mirror Symmetry for del Pezzo surfaces}

Here we put our explicit examples into the broader context of Mirror Symmetry for del Pezzo surfaces.
Our presentation is  necessarily limited and  may not include all relevant citations.\footnote{We apologise in advance for any omission.}

The Fano/Landau--Ginzburg (LG) correspondence predicts that the mirror of a Fano orbifold $X$, i.e.\ a canonical stack whose coarse moduli space is a Fano variety with quotient singularities,
is an LG model, i.e.\ a pair $(M,w)$, where $M$ is a non-compact manifold (carrying a complex and a symplectic structure) and $w$ is a complex-valued function on $M$ called (super)potential.
At a categorical level, a formulation of the correspondence\footnote{A parallel formulation, translating the Hodge-theoretic version of Mirror Symmetry given here, is an equivalence between the Fukaya category of the Fano variety and the category of matrix factorisations of $(M,w)$, see \cite[Remark 1.2]{AKO} and the references therein.} predicts an equivalence between the bounded derived category of coherent sheaves on $X$
 and a suitable analog of the Fukaya category for the symplectic fibration $w \colon M \to \CC$ 
 -- we refer the reader to \cite{AKO, AKOwps, auroux_Tduality,  hacking_keating} and to the references therein.
At a Hodge-theoretic level, 
the correspondence
is interpreted as 
an identity between 
two cohomological invariants:  the regularised quantum period of $X$ (which is a generating function for certain genus-$0$ Gromov--Witten invariants of $X$), and a distinguished period of the mirror $(M,w)$ -- see \cite{victor_calabiyau_compactifications, victor_compactifications, MR3430830, CCGK, CCGG, sigma, victor_LG,victor_toric_LG, victor_models, katzarkov_victor, victor_degenerations, golyshev, mandel_fano_mirror, alessio_cluster} and the references therein.

When an $n$-dimensional Fano orbifold $X$ has a $\QQ$-Gorenstein  (qG) degeneration to a toric variety, it is expected that the LG model $M$ is covered by open subsets isomorphic to $(\CC^*)^n$;
the restriction of $w$ to each torus chart gives a Laurent polynomial.
We say that a Fano variety $X$ is \emph{mirror} to a Laurent polynomial $g$ if the regularised quantum period of $X$ coincides with the classical period of $g$.

This is one of the most 
straightforward definitions of Mirror Symmetry for Fano varieties.
For related and/or more refined versions, in the case of smooth del Pezzo surfaces, we refer the reader to the works \cite{AKO, carl_pumperla_siebert}, which construct and study LG mirrors with proper potentials, and to \cite{proper_LG_potential_open_mirror_map}, which studies the relation between the infinite torus charts of the LG model -- see also \cite{thomas_P2} in the case of $\PP^2$, \cite{ruddat_siebert, ruddat_tropical_cycle}  for more insights into the tropical geometry related to Mirror Symmetry, and \cite{ghk, ghks, gs_annals, gs_intrinsic, gs_intrinsic_utah, mandel_theta} and the references therein for more insights into the algebraic geometry related to Mirror Symmetry.





The relation between degenerations of smooth del Pezzo surface to toric Gorenstein del Pezzo surfaces and Mirror Symmetry is explained in the following result, which is contained in the cited works above in different flavours and here is stated in the language and notation used in this article:

\begin{theorem}
	Let $P$ be a reflexive polygon in the rank $2$ lattice $N$. Let $X_P$ (resp.\ $Y_P$) be the $T_N$-toric (resp.\ $T_M$-toric) del Pezzo surface associated to the face (resp.\ normal) fan of $P$.
	\begin{itemize}
		\item Let $X$ be a general smoothing of $X_P$; so $X$ is a smooth del Pezzo surface with very ample anticanonical class.
		\item Let $Y \to Y_P$ be the blowup described in Construction~\ref{constr:2} and let $Y \to \PP^1$ be the elliptic fibration.  Let $f_P \in \CC[N]$ be the Laurent polynomial, with Newton polytope $P$, discussed in \S\ref{sec:Laurent-fP}.
	\end{itemize}
	Then $f_P$ is mirror to $X$.
\end{theorem}

So, the mirror of (the deformation family of) the smoothings of the (possibly singular) toric del Pezzo surface $X_P$ is an open part of the elliptic fibration $f \colon Y \to \PP^1$ constructed in Construction~\ref{constr:2} from $P$.
One could also see that the $8$ mutation-equivalence classes of reflexive polygons $1$-to-$1$ correspond to the $8$ deformation families of smooth del Pezzo surfaces with very ample anticanonical class (namely, $\PP^1 \times \PP^1$ and the blowup of $\PP^2$ in at most $6$ points).

\begin{example}[{Interpretation of reflexive polygons of volume $4$}]
	We saw that $P_{4a}$ is mutation-equivalent to $P_{4c}$, but not to $P_{4b}$.
	Indeed, $\PP(1,1,2)=X_{P_{4c}}$ (the quadric cone) deforms to $\PP^1 \times \PP^1=X_{P_{4a}}$ (the quadric surface), but not to $\FF_1=X_{P_{4b}}$.
\end{example}

In general, one can prove that two del Pezzo surfaces which are associated to the face fan of mutation-equivalent polygons are actually deformation-equivalent (see \cite{ilten_sigma} and also \cite{petracci_homogeneous}).

One can also treat Mirror Symmetry for smooth del Pezzo surfaces whose anticanonical class is not very ample: there is a toric degeneration to a non-Gorenstein toric surface and the mirror is related to a non-reflexive polygon \cite{MR3430830   ,  wendelin}.
The same is true for del Pezzo surfaces with cyclic quotient singularities which admit toric degenerations.
If there is no toric degeneration, there are no polygons involved and no systematic mirror construction exists. An ad hoc mirror construction for an explicit family of del Pezzo surfaces without a toric degeneration is given in \cite{corti_gugiatti}.

\bibliography{Reflexive_polygons_biblio}

\end{document}